\documentclass[11pt]{amsart}
\usepackage[english]{babel}
\usepackage[T1]{fontenc}

\usepackage{amsmath}
\usepackage{amsthm}
\usepackage{amssymb}
\usepackage[all]{xy}
\usepackage{mathrsfs}
\usepackage{enumerate}
\usepackage{bbm,dsfont}
\usepackage{bm}
\usepackage{color}

\usepackage{graphicx}
\usepackage{float}
\usepackage{url}

\newtheorem{theorem}{Theorem}[section]
\newtheorem{problem}[theorem]{Problem}
\newtheorem{lemma}[theorem]{Lemma}
\newtheorem{proposition}[theorem]{Proposition}

\newtheorem{definition}[theorem]{Definition}
\newtheorem{prob}[theorem]{Problem}
\theoremstyle{definition}
\newtheorem{remark}[theorem]{Remark}

\newtheorem{example}[theorem]{Example}

\newcommand{\w}{\omega}

\usepackage{xcolor}

\newcommand{\balpha}{{\bm{\alpha}}}
\newcommand{\bz}{{\bm{z}}}
\newcommand{\bw}{{\bm{w}}}

\newcommand{\rom}[1]{\expandafter{\romannumeral #1\relax}}
\newcommand{\RNum}[1]{\uppercase\expandafter{\romannumeral #1\relax}}

\newcommand{\ra}{\rangle}

\newcommand{\llangle}{\langle \kern -0.2em \langle}	
\newcommand{\rrangle}{\rangle \kern -0.2em \rangle}

\DeclareMathOperator*{\argmin}{arg\,min}

\author[N. Alpay]{Natanael Alpay}
\address{(NA)
Department of Mathematics\\ 
Universiy of California Irvine,
Irvine, CA 92697 \\
USA}
\email{nalpay@uci.edu}

\author[A.De Martino]{Antonino De Martino}
\address{(ADM) Politecnico di
	Milano\\Dipartimento di Matematica\\Via E. Bonardi, 9\\20133 Milano\\Italy}
\email{antonino.demartino@polimi.it}

\author[K. Diki]{Kamal Diki}
\address{(KD) Clifford Research Group, Department of Electronics and Information Systems, Faculty of Engineering and Architecture, Ghent University, Krijgslaan 281, 9000 Gent, Belgium.}
\email{Kamal.Diki@UGent.be}

\title[]{Representer theorem in complex reproducing kernel Hilbert spaces with applications to Fock and Hardy spaces and superoscillations}

\begin{document}
\begin{abstract}
We introduce a complex-valued counterpart of the representer theorem in machine learning. We study several learning and minimization problems in reproducing kernel Hilbert spaces (RKHSs), with the aim of identifying appropriate input-output data sets that allow specific functions to appear as solutions of regression-type minimization problems. In particular, we recover superoscillations in the Fock space, the Gaussian radial basis function (RBF) kernel in the corresponding RKHS, and finite Blaschke products in the Hardy space setting. We then extend the notion of superoscillations through suitable generalizations of the Fock space and investigate the associated learning problems.
This is a seminal work relating superoscillations and machine learning kernel methods via the representer theorem.
\end{abstract}
\maketitle
\tableofcontents



\section{Introduction}

Reproducing kernel Hilbert spaces (RKHS) provide a natural meeting point for functional analysis, complex analysis, approximation theory, and statistical learning. We refer to \cite{PR2016, saitoh} for a general introduction to the theory of RKHSs, and to  \cite{BerlinetThomasAgnan2004} for a probability and statistics perspective. The starting point is that a positive definite kernel can be realized through the inner product of a Hilbert space of functions, a fact going back to the foundational work of Aronszajn; see \cite{Aronszajn1950, BerlinetThomasAgnan2004,saitoh}. \\ \\ From the point of view of applications, this observation leads to the so-called kernel trick: nonlinear dependence on the input variables can be handled through linear operations in a suitable feature space, without requiring an explicit parametrization of that space. This principle lies behind a large class of kernel-based methods in data analysis and learning theory, including kernel ridge regression (also known as Tikhonov regularization) and support vector machines; see \cite{ScholkopfHerbrichSmola2001, SVM}. A second key principle is the representer theorem. In its classical form, it asserts that a minimizer of a regularized empirical risk functional over an RKHS belongs to the finite-dimensional span of the kernel functions centered at the sample points. In this way, an infinite-dimensional variational problem is reduced to a finite-dimensional one.
This mechanism was introduced by Kimeldorf and Wahba in 1971, and has since become a standard tool in kernel methods; see \cite{KimeldorfWahba1971, ScholkopfHerbrichSmola2001, Wahba1990}
. Although these ideas are often presented in the real-valued setting, many spaces arising naturally in analysis and mathematical physics are intrinsically complex. This is the case, for instance, for the Fock space of entire functions and for the Hardy space on the unit disk. \\ \\ 

The present paper is motivated by the interaction between this RKHS point of view and the theory of superoscillations. Roughly speaking, a superoscillating sequence is built from elementary waves whose frequencies remain uniformly bounded, while the sequence converges on compact sets to a wave with a larger frequency. Originating in quantum mechanics, superoscillations have developed into a substantial analytic theory with applications to Schr\"odinger evolution, weak values, and related approximation problems; see \cite{SupBook, BZACSSTRQHL2019}. A connection between superoscillations and the Fock space was established in \cite{ACDSS}, where the authors observed, in \cite[Remark 4.2]{ACDSS}, the potential to interpret the classical superoscillation sequence as a specific solution to a learning problem involving the representer theorem. Our aim here is to show that this representer-theorem perspective indeed provides a convenient and systematic way to recover such functions as solutions of regression-type minimization problems, including the classical superoscillations. \\ \\

The paper has two main objectives. The first is to formulate and prove a representer theorem in the complex setting adapted to the classes of kernels considered here. The second is to use this result in a reverse learning direction: rather than starting from prescribed data and seeking the minimizer, we fix a function in a complex RKHS and determine input-output data for which that function is the minimizer of a regularized regression problem. This procedure is carried out in several situations. In the Fock space we recover classical superoscillations sequence functions and related Gaussian radial basis function (RBF) constructions. We then pass to various extensions of Fock-type spaces. For instance, we obtain supershift analogues associated with the Mittag-Leffler kernel function considered in \cite{mlf}. The case of other generalized Fock spaces involving Touchard polynomials that were studied recently in \cite{ADK} will be considered in this paper. We also investigate the case of the Hardy space on the disk and identify data for which a finite Blaschke product solves the corresponding minimization problem. Finally, we consider Cauchy evolution problems, including the free Schr\"odinger evolution generated by RBF-type superoscillatory data. \\ \\

The paper is organized as follows. Section~2 reviews some basic notions and properties about reproducing kernel Hilbert spaces, Fock spaces, and superoscillations.  In Section~3, we prove the complex representer theorem and derive the form of the output data in the regression setting. Section~4 is devoted to the Fock space and its relation with superoscillations. Section~5 treats supershift constructions in generalized RKHS. In Section~6, we discuss the Hardy-space setting and the case of finite Blaschke products. Finally, Section~7 studies the free-particle Schr\"odinger equation with RBF superoscillatory initial data.

\section{Preliminaries}
We will briefly review the properties of reproducing kernel Hilbert spaces needed in the following sections. We first recall the definition:

\begin{definition}
A reproducing kernel Hilbert space is a Hilbert space $(\mathcal H,\langle\cdot,\cdot\rangle)$ of functions defined in a non-empty set $\Omega$, such that there exists a complex-valued function $K(z,\omega)$ defined
on $\Omega\times\Omega$ and with the following properties:
	\begin{enumerate}
		\item $z\mapsto K(z,\omega)\in \mathcal H$ for any $ \omega \in \Omega$
		\item $ \langle f,K_\omega\ra=f(\omega)$, for any $f \in \mathcal{H}$.
	\end{enumerate}
\end{definition}

The function $K(z, \omega)$ is uniquely determined by the Riesz representation theorem and is called the reproducing kernel of the space. 
The reproducing kernel (or simply, the kernel) has a very important property: it is positive definite. 
That is, for all $N \in \mathbb{N}$, $\omega_1, \dots, \omega_N \in \Omega$, and $c_1, \dots, c_N \in \mathbb{C}$, we have
\[
\sum_{i=1}^{N} \sum_{j=1}^{N} c_i \overline{c_j}\, K(\omega_i, \omega_j) \geq 0.
\]

We refer to the book \cite{saitoh} for more information on reproducing kernel Hilbert spaces, and we recall that there is a one-to-one correspondence between positive definite functions on a given set
and reproducing kernel Hilbert spaces of functions defined on that set.

An example of reproducing kernel Hilbert space, that we will use in this paper is given by the Fock space, see \cite{zhu}.
\begin{definition}\label{def:fock}
	The Fock (or Segal-Bargmann-Fock) space $\mathcal F$ consists of entire functions $f$ that are square integrable with respect to the Gaussian measure, i.e. 
	\begin{equation}
		\label{gauss123}
		\mathcal{F}(\mathbb{C}) = \left\{ f\in \mathcal{H}(\mathbb{C}):\frac{1}{\pi}\int_{\mathbb C}|f(z)|^2e^{-|z|^2}dA(z)<\infty\right\},
	\end{equation}
	where $\mathcal{H}(\mathbb{C})$ denotes the space of entire functions, and $dA(z)=dxdy$ is the Lebesgue measure on $\mathbb{C}$ with $z=x+iy$.
\end{definition}

The reproducing kernel of the Fock space is given by
\begin{equation}
	\label{fock}
	B(z,w) =  e^{z\overline{\w}},\quad z,w\in \mathbb{C} .
\end{equation}
We remark that, up to a positive multiplicative factor, the Fock space is the unique Hilbert space of entire functions in which
\begin{equation}
	\label{fock-equation}
	\partial_z^*=M_z,
\end{equation}
where $\partial_z$ denote the derivative with respect to $z$. Now we review the basic notions of superoscillations; for further details, we refer the reader to \cite{SupBook}.

\begin{definition}
We call generalized Fourier sequence a sequence of the form
\begin{equation}
f_n(x)= \sum_{j=0}^{n}Z_j(n,a) e^{i h_j(n)x}
\end{equation}
where $a \in \mathbb{R}$, $n \in \mathbb{N}$, $Z_j(n,a)$ is complex-valued, and $h_j(n)$ is real-valued.
\end{definition}

\begin{remark}
The sequence of partial sums of a Fourier expansions is a particular case of the above notion with $Z_j(n,a)=Z_j \in \mathbb{R}$ and $h_j(n)=h_j \in \mathbb{R}$ are multiples of a real number.
\end{remark}

\begin{definition}
A generalized Fourier sequence 
\begin{equation}
	\label{Fourierg}
	f_n(x) = \sum_{j=0}^{n} Z_j(n,a)\, e^{i h_j(n) x}, \qquad n \in \mathbb{N}, \quad a \in \mathbb{R},
\end{equation}
is said to be a superoscillating sequence if 
\begin{itemize}
\item $|h_j(n)| \leq 1$ for all $n\in\mathbb{N}$ and $j \in \mathbb{N}_0$
\item   there exists a compact subset of $\mathbb{R}$, referred to as a superoscillation set, on which $f_n(x)$ converges uniformly to $e^{i g(a) x}$, where $g$ is a continuous real-valued function satisfying $|g(a)| > 1$.
\end{itemize}

\end{definition}

\begin{remark}
The term superoscillations arises from the fact that in the Fourier representation \eqref{Fourierg}, the frequencies $h_j(n)$ are bounded by $1$, while the limit function $e^{ig(a)t}$ has frequency $g(a)$, which can be arbitrarily larger than $1$. 
\end{remark}

The classical Fourier expansion is clearly not a superoscillating sequence, as its frequencies are generally unbounded.

\begin{example}
\label{super}
The prototypical example of a superoscillating sequence is the following:
\begin{equation}
	\label{reps}
	F_n(x,a)= \left(\cos \left(\frac{x}{n}\right)+ia \sin \left(\frac{x}{n}\right)\right)^n= \sum_{j=0}^{n}C_j(n,a)e^{i\left(1-\frac{2j}{n}\right)x}, \quad n \in \mathbb{N}, \quad x \in \mathbb{R},
\end{equation}
where $a>1$ and the coefficients $C_j(n,a)$ are given by
\begin{equation}\label{eq:cjzj}
	C_j(n,a) = {n\choose j} \left(\frac{1+a}{2}\right)^{n-j} \left(\frac{1-a}{2}\right)^j.
\end{equation}
If we fix $x \in \mathbb{R}$ and se tend $n$ to infinity, we obtain that
$$ \lim_{n\to\infty } F_n(t,a)= e^{iat},$$
and the limit is uniform on compact subsets of the real line.
\end{example}

\section{Complex Representer Theorem and Regression Minimization}

\subsection{General Setting}

Let $\Omega$ be a non empty set of $\mathbb{C}$, and the traning data set $\{(z_j,w_j)\}_{j=0}^{n}\subseteq \Omega \times \Omega$.
Denote 
$\bm{w}=(w_0,\dots, w_n)^T\in \mathbb{C}^{n+1\times 1}$,
and 
$\bm{z} = (z_0,\dots,z_n)^T\in \mathbb{C}^{n+1\times 1}$.
Let $K:\Omega\times \Omega \to \mathbb{C}$ be a positive definite kernel, and $\mathcal{H}_K$ be the associated RKHS with norm $\|\cdot\|_{\mathcal{H}_K}$, i.e. for any $z_j \in \Omega , \beta_j \in \mathbb{C}$,
$$
\left\|\sum_{i=0}^{\infty} \beta_i K\left(\cdot, z_i\right)\right\|^2=\sum_{i=0o}^{\infty} \sum_{j=0}^{\infty} \beta_i \bar{\beta}_j K\left(z_i, z_j\right) .
$$
The representer theorem is usually stated in the real case, see for example \cite[Thm 5.5]{SVM}. However, for the purposes of this paper, we require a representer theorem in the complex setting. Since we could not find a proof of this result in the literature, we provide one below.

\begin{theorem}[Complex Representer Theorem]\label{thm:crt}
     Let
     $\mathcal{H}_K$ be the RKHS associated to $K$, with norm $\|\cdot\|_{\mathcal{H}_K}$, 
     and $g$
     a strictly monotonically increasing real-valued function on $[0, \infty)$.
Let us consider the functional  $J:\mathcal{H}_K\to \mathbb{R}$ defined by

\begin{equation}
\label{eq:J}
J(f)
= L_{\bm{w}}\left(f\left(z_0\right), \ldots,f\left(z_n\right)\right)
+
g(\|f\|_{\mathcal{H}_K}), \quad \text{ for all } f\in \mathcal{H}_K,
\end{equation}
where $L_{\bm{w}}:\mathbb{C}^{n+1} \rightarrow\mathbb{R} \cup\{\infty\}$ is an arbitrary loss function. Then, any solution $f^*$ of the optimization problem
\[
f_*= \underset{f\in\mathcal{H}_K}\argmin ~  J(f),
\]
has the form
\begin{equation}\label{eq:fmin}
f_*(\cdot)=\sum_{k=0}^n \alpha_k K\left(\cdot, z_k\right) .    
\end{equation}

\end{theorem}

\begin{proof}
Let $ J(f) $ be the functional to be minimized, and define $ \mathcal{H}_{\mathcal{S}} $ as the linear span in $ \mathcal{H}_K$ of the vectors $S=\{K_{z_j}\}_{j=0}^{n} $,
given by
\[ 
\mathcal{H}_{\mathcal{S}} = \left\{ f \in \mathcal{H} : f(z) = \sum_{k=0}^n \alpha_k K\left( z, z_k \right), \alpha_0, \cdots, \alpha_n \in \mathbb{C}\right\}.
\]

By the orthogonal decomposition theorem, we have $ \mathcal{H} = \mathcal{H}_{\mathcal{S}} \oplus \mathcal{H}_{\mathcal{S}}^{\perp} $. Therefore, any function $ f \in \mathcal{H} $ can be uniquely decomposed as
\[
f = f_{\mathcal{S}} + f_{\mathcal{S}}^{\perp},
\]
where $ f_{\mathcal{S}} \in \mathcal{H}_{\mathcal{S}} $ and $ f_{\mathcal{S}}^{\perp} \perp \mathcal{H}_{\mathcal{S}} $ (by orthogonal projection). We aim to find $ f^* = \arg\min_{f \in \mathcal{H}_K} J(f) $, i.e., $ J(f^*) \leq J(f) $ for all $ f \in \mathcal{H} $. We note that $f_\mathcal{S}$ and $f_{\mathcal{S}}^{\perp}$ are orthogonal, thus 
   \[
   \| f \|^2 = \| f_{\mathcal{S}} \|^2 + \| f_{\mathcal{S}}^{\perp} \|^2 \geq \| f_{\mathcal{S}} \|^2.
   \]
   Now, since $ f_{\mathcal{S}}^{\perp} \in \mathcal{H}_{\mathcal{S}}^{\perp} $ and $ K(z_i, \cdot) \in \mathcal{H}_{\mathcal{S}} $, by the reproducing kernel property we get 
   \[
   f_{\mathcal{S}}^{\perp} (z_k) = \langle f_{\mathcal{S}}^{\perp}, K(\cdot,z_k) \rangle = 0, \qquad  \text{ for all } k = 0, \dots, n.
   \]
Finally, since $g$ is strictly increasing we have for any $f\in \mathcal{H}_K$	
\begingroup\allowdisplaybreaks
   \begin{align*}
       J(f) &= L_\bw \left( f_{\mathcal{S}} (z_0) + f_{\mathcal{S}}^{\perp} (z_0), \cdots, f_{\mathcal{S}} (z_n) + f_{\mathcal{S}}^{\perp} (z_n) \right) + g (\| f \|^2) \\
       &\geq L_\bw \left( f_{\mathcal{S}} (z_0) + f_{\mathcal{S}}^{\perp} (z_0), \cdots, f_{\mathcal{S}} (z_n) + f_{\mathcal{S}}^{\perp} (z_n) \right) + g (\| f_{\mathcal{S}} \|^2) \\
       &= L_\bw \left( f_{\mathcal{S}} (z_0), \cdots, f_{\mathcal{S}} (z_n) \right) + g (\| f_{\mathcal{S}} \|^2) \\
       &= J(f_{\mathcal{S}}).
   \end{align*}
   \endgroup

Hence, the minimizing function $f_*$ lies in $ \mathcal{H}_{\mathcal{S}} $ and has the form $$f_*(z)= f_{\mathcal{S}}(z )= \sum_{k=0}^{n} \alpha_k K(z,z_k).$$
\end{proof}

\subsection{Regression Setting}
Let $\Omega$ be a non empty set of $\mathbb{C}$, and the traning data set $\{(z_j,w_j)\}_{j=0}^{n}\subseteq \Omega \times \Omega$.
Let $K$ be a kernel associated with the RKHS given by $\mathcal{H}_K$ and $f \in \mathcal{H}_K$. In the case of the regression the loss function and  the regularizer are given by

\[
L_{\bw}(f(z_0),\dots, f(z_n)) = \sum_{j=0}^{n}|w_j-f(z_j)|^2 , \qquad g(\|f\|^2)=\lambda \|f\|^2,
\]
for some $\lambda>0$. Thus, the empirical risk functional in this case is defined by
\begin{equation}\label{eq:jf}
    J(f) =L_{\bw}(f(z_0),\dots, f(z_n)) +\lambda \|f\|^2=
\sum_{j=0}^{n}|w_j-f(z_j)|^2  +\lambda \|f\|^2.
\end{equation}

According to the complex representer theorem, see Theorem \ref{thm:crt}, a function $f \in \mathcal{H}_K$ that minimizes the functional $J(f)$ is given by

$$
f_*(z) = \sum_{j=0}^{n} \alpha_j K(z, z_j), 
$$
Now, since $f \in \mathcal{H}_K$ we can write $f=h_1+h_2$, where $h_1 \in \mathcal{H}_S$ and $h_2 \in \mathcal{H}_S^{\perp}$. Thus we  have
$$ f(z)= \langle f, K_z \rangle= \langle h_1, K_z \rangle+ \langle h_2, K_z \rangle= h_1(z)=\sum_{j=0}^{n} \alpha_j K(z,z_j).$$
Therefore, by \eqref{eq:fmin} we get
\begin{equation}
\label{onestar}
f(z_k)=f_{*}(z_k)=\sum_{j=0}^n \alpha_j K(z_k, z_j).
\end{equation}
This implies that we can write the functional $J(f)$ as
\begin{equation}
\label{functional}
J(f) = \sum_{k=0}^{n} \left| w_k - \sum_{j=0}^{n} \alpha_j K(z_k, z_j) \right|^2 + \lambda \|f\|^2.
\end{equation}

Now, our goal is to find the $\bw = (w_0, \ldots, w_n)^T\in \mathbb{C}^{(n+1) \times 1}$ that minimizes the above functional.

\begin{lemma}
\label{fun}
For a fixed ${\bm{w}}\in \mathbb{C}^{(n+1) \times 1}$ and for $\balpha=(\alpha_0, \ldots, \alpha_n)^T \in \mathbb{C}^{(n+1) \times 1}$ the functional $J(f)$ can be rewritten as
\begin{equation}
	\label{functional1}
J(f) = (\bw - K\balpha)^* ( \bw - K \balpha) + \lambda \balpha^* K \balpha,
\end{equation}
where $K \in \mathbb{C}^{(n+1) \times (n+1)}$ is a symmetric matrix with entries $K(z_k, z_j)$ for $0 \leq k, j \leq n$ and $(.)^*$ denotes the conjugate-transpose matrix. 
\end{lemma}
\begin{proof}
By \eqref{onestar}, we observe that the function $f_*$ can be expressed as the product of the matrix $K$ and the vector $\boldsymbol{\alpha}$, namely
\begin{equation}
	f_* = K \boldsymbol{\alpha}.
\end{equation}

This implies that
\begin{equation}
\label{p2}
\|f_*\|^2= \balpha^*K\balpha.
\end{equation}
By using the definition of the vector $\bm{w}$ and the matrix $K$ we have
		\begingroup\allowdisplaybreaks
\begin{eqnarray}
	\nonumber
\sum_{k=0}^{n} \left| w_k - \sum_{j=0}^{n}\alpha_j K(z_k,z_j) \right|^2&=&\sum_{k=0}^{n} \left[
|w_k|^2 -2\text{Re}\left( \bar{w}_k \sum_{j=0}^{n} \alpha_j K(z_k,z_j)  \right)+ \left|\sum_{j=0}^{n} \alpha_j K(z_k,z_j) \right|^2
\right]\\
\nonumber
&=& \sum_{k=0}^{n} 
\left[
|w_k|^2 
-2\text{Re}\left(\bar{w}_k\sum_{j=0}^{n} \alpha_j K(z_k,z_j)   \right)
\right.\\
\nonumber
&&\left.+\sum_{s=0}^{n}\sum_{j=0}^{n} \alpha_s\bar{\alpha}_j \overline{K(z_i,z_j)}K(z_k,z_s)
\right]\\
\nonumber
&=&\sum_{k=0}^{n}\left[ |w_k|^2 - \bar{w}_k \sum_{j=0}^{n}\alpha_j K(z_k,z_j) \right.\\
\nonumber
&&\left. - w_k \sum_{j=0}^{n} \bar{\alpha}_j \overline{K(z_k,z_j) }
+ \sum_{s=0}^{n} \sum_{j=0}^{n} \alpha_s \bar{\alpha}_j \overline{K(z_k,z_j)}K(z_k,z_s)
\right] \\
\nonumber
&=& (\bw-K\balpha)^*(\bw-K\balpha)\\
\label{comp1}
&=&|\bm{w}-K\balpha|^2.
\end{eqnarray}
\endgroup
Finally, by plugging \eqref{p2} and \eqref{comp1} into \eqref{functional} we get \eqref{functional1}.
\end{proof}

\begin{remark}
Since the functional $J(f)$ depends only on $\alpha$, from now on we will use the notation $J(\balpha)$ to denote the functional $J(f)$.
\end{remark}

To minimize $J(\balpha)$, we need to differentiate with respect to $\balpha$ and solve for $\nabla_\balpha J(\balpha) = 0$. For this reason we recall some basic properties of the complex gradient operator studied in \cite[Section 3]{Brandwood}.
\begin{lemma}
	\label{rules}
Let $A \in \mathbb{C}^{n \times n}$ be a complex matrix.
For a complex vector $\bz=(z_1,\cdots , z_n)^T \in \mathbb{C}^{n\times 1}$ the
complex gradient is given by $\nabla_\bz =(\partial_{z_1}, ,\cdots , \partial_{z_n})^T$.
We have the following properties:
\begin{enumerate}[(i)]
\item $\nabla_\bz (A^* \bz) = \bar{A}$. 
\item $\nabla_\bz (\bz^* A) =0.$    
\item $\nabla_\bz (\bz^* A \bz ) = A^T \bar{\bz}.$
\item $\nabla_\bz (\bz^* A\bz ) = \bar{A}\bar{\bz },  $ for $A$ Hermitian.
\end{enumerate}
\end{lemma}

\begin{proposition}\label{prop:w}
Let $\bw = (w_0, \ldots, w_n)^T\in \mathbb{C}^{(n+1) \times 1}$, $\balpha= (\alpha_0, \ldots, \alpha_n)^T\in \mathbb{C}^{(n+1) \times 1}$ and $K \in \mathbb{C}^{(n+1) \times (n+1)}$ being a strictly positive definite (Hermitian) matrix with entries $K(z_i, z_j)$ for $0 \leq i, j \leq n$. Then $ J(\balpha)$ is minimized 
if and only if

\begin{equation}
\label{eq:w}
    \bw = (K+\lambda I)\balpha.
\end{equation}

\end{proposition}

\begin{proof}
By Lemma \ref{fun} we can write the functional $J(\balpha) $ as
$$J(\balpha)=	\bw^* \bw - \bw^* K \balpha - \balpha^* K^* \bw + \balpha^* K^2\balpha+ \lambda \balpha^* K \balpha.$$
We observe that the kernel $K=( K(z_i,z_j))_{0\leq i,j\leq n}  $ is Hermitian, indeed we have $K^*=K$.\\

We apply the complex gradient $\nabla_\balpha  = (\partial_{\alpha_{1}}, ,\cdots , \partial_{\alpha_n})^T$ to $J(\balpha)$, by Lemma \ref{rules} get,
		\begingroup\allowdisplaybreaks
\begin{align*}
    \nabla_\balpha J(\balpha) 
&=
\nabla_\balpha( \bw^*\bw )- \nabla_\balpha( \bw^* K \balpha) -  \nabla_\balpha( \balpha^* K^* \bw) + \nabla_\balpha (\balpha^* K^2 \balpha)+ \lambda \nabla_\balpha( \balpha^* K \balpha )\\
&=-(\bw^*K)^T+\bar{K}^2 \bar{\balpha } + \lambda \bar{K}\bar{\balpha }\\
&= -\bar{K} \bar{\bw} +\bar{K}^2 \bar{\balpha } + \lambda \bar{K}\bar{\balpha }.
\end{align*}
\endgroup
In order to minimize $J(\balpha)$ we have to solve $\nabla_\balpha  J(\balpha )=0$ for $\bw$. Since the Gram matrix $K=(K(z_i,z_j))_{i,j}$ is strictly positive definite, it is invertible (see \cite{SS2002}) and therefore we have $\nabla_\balpha  J(\balpha )=0$ if and only if
\[
\bw = (K+\lambda I)\balpha.
\]

\end{proof}

\section{Representer Theorem and Superoscillations}

In this section, we provide a bridge between the complex Representer theorem and superoscillations. We will see direct applications on the Fock and Gaussian RBF spaces.

\subsection{Fock Space}
Here our goal is to establish a connection between the complex Representer theorem and superoscillations. For this reason, we consider the holomorphic extension of the sequence $f_n(x)$ defined in \eqref{Fourierg} to entire functions by replacing the real variable $x$ with the complex variable $z$. Moreover, using the definition of the reproducing kernel of the Fock space, we can express the sequence $f_n(z)$  in terms of the reproducing kernel of the Fock space as follows:
\begin{equation}
\label{gensuper}
	f_n(z)= \sum_{j=0}^{n} Z_j(n,a) B(z,-ih_j(n)).
\end{equation}
where the function $B(z, -ih_j(n))$ is defined in \eqref{fock}.

\begin{remark}
We can also consider the holomorphic extension of the example given in \eqref{reps} by replacing the real variable $x$ with the complex variable $z$. In this case as well, the superoscillations in \eqref{reps} can be expressed in terms of the reproducing kernel of the Fock space:

\begin{equation}\label{eq:suposc}
	F_n(z,a)=\sum_{j=0}^{n} C_j(n,a)B(z,z_j), \quad z_j  = -i \left(1-\frac{2j}{n}\right).
\end{equation}
\end{remark}

We now state the following learning problem for the Fock space.

\begin{prob}
\label{p1}
Let $n\in\mathbb{N}, (h_j(n))_{0\leq j\leq n}$ be a real-valued sequence such that $\sup_{0 \leq j \leq n}|h_j(n)| \leq 1$.
For an input set of the form $ \{z_j\}_{j=0}^n= \biggl \{-i h_j(n) \biggl\}_{j=0}^n$, we seek a corresponding output set $w_0, \dots, w_n$ such that the superoscillating sequence $f_n(z)$, see \eqref{gensuper}, solves the learning problem
\begin{equation}
\label{min}
f_* = \arg\min_{f \in \mathcal{F}(\mathbb{C})} J(f),
\end{equation}
where
$$
J(f) = \sum_{k=0}^n \left| w_k -f(z_k) \right|^2 + \lambda \|f\|^2, \qquad \lambda>0, \qquad f\in \mathcal{F}(\mathbb C).
$$
\end{prob}

\begin{theorem}
Let $n \in \mathbb{N}$ and $a>1$. The components of the vector  $\bw= \{w_k\}_{k=0}^n$ that solves the problem \ref{p1} are given by
\begin{equation}
\label{compp}
w_k= \sum_{j=0}^{n}Z_j(n,a) e^{h_j(n)h_k(n)}+\lambda Z_k(n,a),
\end{equation}
where the sequence $h_j(n)$ is real-valued and is defined such that $|h_j(n)| \leq 1$ for all $n$ and $j \in \mathbb{N}_0.$
\end{theorem}
\begin{proof}
By the complex representer theorem (see Theorem \ref{thm:crt}), the minimizer $f_*$ is given by
$$ f_*(z)= \sum_{j=0}^{n} C_j(n,a) B(z,z_j).$$
We take as input set $ \{z_j\}_{j=0}^n= \biggl \{-i h_j(n) \biggl\}_{j=0}^n$. By Proposition \ref{prop:w} we know that the functional $J(f)$ is minimized if and only if 
\begin{equation}\label{res}
	\bw=(K+\lambda I)\balpha.
\end{equation}

where

$$
\balpha = (Z_0(n,a),\dots, Z_n(n,a))^T, \qquad K=( B(z_k,z_j) )_{0\leq k,j\leq n} = \left( e^{h_j(n) h_k(n) }\right)_{0\leq k,j\leq n}.
$$

The components of $K\balpha$ are given by
\begin{equation}
\label{prod}
K \alpha_k= \sum_{j=0}^{n} Z_j(n,a) e^{h_j(n)h_k(n)}.
\end{equation}
Thus by plugging \eqref{prod} in \eqref{res} we have the result.
\end{proof}

The components that solve Problem \eqref{p1} can take a more compact form if we consider specific examples of superoscillations.

\begin{theorem}
\label{class}
Let $n \in \mathbb{N}$ and $a > 1$. If, in Problem \ref{p1}, we consider the superoscillating sequence \eqref{eq:suposc} as the solution of the minimization problem \eqref{min}, then the components of the vector $\mathbf{w} = \{w_k\}_{k=0}^n$ that solve Problem \eqref{p1} are given by

\begin{equation}
\label{comp}
w_k(n,a) = \left[ \frac{1+a}{2}\right]^n \left[ e^{1-\frac{2k}{n}}\left(1 +\frac{1-a}{1+a}e^{-\frac{2}{n} +\frac{4k}{n^2} }  \right)^n +\lambda {n\choose k} \left(\frac{1-a}{1+a}\right)^k    \right].
\end{equation}
\end{theorem}

\begin{proof}
Since we are considering the superoscillations $F_n(z,a)$ given in \eqref{eq:suposc} as a solution of the minimization problem, it is sufficient to use formula \eqref{compp} with $Z_\ell(n,a) = C_\ell(n,a)$ for $\ell = j,k$, and $h_\ell(n) = 1 - \frac{2\ell}{n}$ where $C_{\ell}(n,a)$ are defined in \eqref{eq:cjzj} for $\ell = j,k$. Thus by formula \eqref{comp} we have

	\begingroup\allowdisplaybreaks
\begin{align*}
\sum_{j=0}^{n} Z_j(n,a) e^{h_j(n)h_k(n)}&= \sum_{j=0}^{n} {n\choose j}\left( \frac{1+a}{2}\right)^{n-j} \left(\frac{1-a}{2}\right)^{j} e^{\left( 1-\frac{2k}{n} \right)\left( 1-\frac{2j}{n}  \right) }  \\
    &= \left(\frac{1+a}{2}\right)^n \sum_{j=0}^{n}{n\choose j} \left(\frac{1-a}{1+a}\right)^j e^{1-\frac{2j}{n} -\frac{2k}{n} +\frac{4kj}{n^2} }\\
    &= e^{1-\frac{2k}{n}}  \left(\frac{1+a}{2}\right)^n \sum_{j=0}^{n}{n\choose j} \left(\frac{1-a}{1+a}\right)^j e^{-\frac{2j}{n} +\frac{4kj}{n^2}}\\
    &= e^{1-\frac{2k}{n}}  \left(\frac{1+a}{2}\right)^n \sum_{j=0}^{n}{n\choose j} \left(\frac{1-a}{1+a}\right)^j \left( e^{-\frac{2}{n} +\frac{4k}{n^2}}\right)^j \\
    &= e^{1-\frac{2k}{n}}  \left(\frac{1+a}{2}\right)^n \left(1+ \left(\frac{1-a}{1+a}\right) e^{-\frac{2}{n} +\frac{4k}{n^2}}\right)^n.
\end{align*}
\endgroup
Since
$
\lambda Z_k(n,a) = \lambda \binom{n}{k} \left(\frac{1+a}{2}\right)^n \left(\frac{1-a}{1+a}\right)^k,
$
the result follows.

\end{proof}

\begin{remark}
For $a=1$, the output values reduce to $$w_k(n,1)=e^{1-\frac{2k}{n}}.$$
\end{remark}

\begin{example}\label{ex:1}
	In the case of regression, the solution to the minimization problem can be interpreted as the slope of the separating hyperplane for the data. In our case, the hyperplane corresponds to the superoscillation sequence.
    to illustrate the solution to Problem \ref{p1},
	consider $n=100$, $a=2$, and $\lambda =1$, we observe how the superoscillation function $F_{100}(z,2)$ separates the data points $\{(z_k, w_k)\}_{k=0}^{100}$, where $z_k=-i \left(1-\frac{2k}{n}\right)$, and $w_k$ is given by \eqref{comp}.
	In Figure~\ref{fig:1}, we plot the data points, with the $x$-axis representing the given inputs $z_k$, and the $y$-axis representing the calculated outputs $w_k$.
	In Figure~\ref{fig:2}, we show the function $F_{100}(z,2)$ plotted over the real line.
	In Figure~\ref{fig:3}, we overlay the two previous visualizations to illustrate how $F_{100}(z,2)$ separates the data points into two classes.
    However, due to the difference in magnitude in the $y$-axis of the data points and the superoscillation, the function $F_{100}(z,2)$ would appear almost linear. To better visualize the separation and behavior of the superoscillation, which presents a log-scale version in Figure~\ref{fig:3}.
	
	\begin{figure}[H]
		\minipage{0.32\textwidth}
		\includegraphics[width=\linewidth]{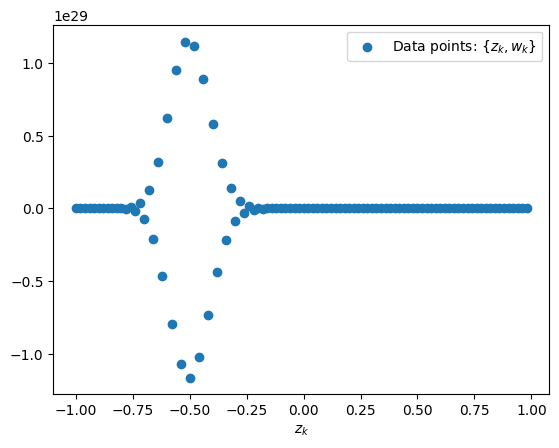}
		\caption{The calculated optimal data $\{w_k\}$ \,
			and the given inputs $\{z_k\}$.
			}
		\label{fig:1}
		\endminipage\hfill
		\minipage{0.32\textwidth}
		\includegraphics[width=\linewidth]{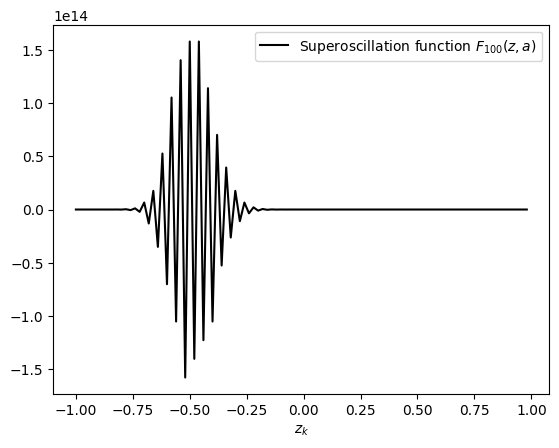}
		\caption{The Superoscillation function $F_{100}(z,a)$ over $\{z_k\}$.}\label{fig:2}
		\endminipage\hfill
		\minipage{0.32\textwidth}
        \includegraphics[width=\linewidth]{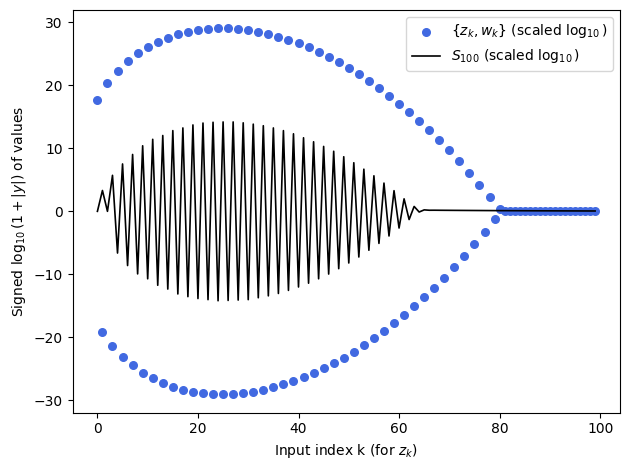}
		\caption{Signed-Log-scaled Superoscillation and the data-point.}\label{fig:3}
		\endminipage
	\end{figure}

\end{example}

\subsection{Gaussian RBF Kernels}

The Radial Basis Function (RBF) kernel, also known as the Gaussian kernel, is a commonly used kernel in various learning algorithms, including support vector machines (SVMs) and Gaussian processes. A detailed study of reproducing kernel Hilbert spaces of holomorphic functions associated with the Gaussian RBF kernel in the setting of several complex variables is provided in \cite{SVM, SD2006}. The connection between the Gaussian RBF kernel and Fock space theory was first explored in \cite{RBF} and has recently been extended to the polyanalytic setting in \cite{DDD}. We recall the following definition.

\begin{definition}[RBF kernel]
    Let $\gamma > 0$, $z \in \mathbb{C}$ and $w \in \mathbb{C}$. The function defined by
\[
K_{\gamma}(z,w) = \exp \left( - \frac{(z - \overline{w})^2}{\gamma^2} \right),
\]
is called the Gaussian RBF kernel with width $\frac{1}{\gamma}$.

\end{definition}

The reproducing kernel Hilbert space (RKHS) corresponding to the complex RBF kernel $K_{\gamma}$ is defined as follows:

\begin{definition}[RBF space]
    Let $\gamma > 0$, an entire function $f : \mathbb{C} \rightarrow \mathbb{C}$ belongs to the RBF space, denoted by $\mathcal{H}_{\gamma}^{RBF}(\mathbb{C})$, if we have
\[
\mathcal{H}_{\gamma}(\mathbb{C}) = 
 \left\{ f\in \mathcal{H}(\mathbb{C}): \left( \frac{2}{\pi \gamma^2} \right) \int_{\mathbb{C}} | f(z) |^2 \exp \left( \frac{(z - \overline{z})^2}{\gamma^2} \right) dA(z) < \infty\right\}.
\]
where $dA(z) = dx\,dy$ is the Lebesgue measure with respect to the variable $z = x + iy$.
\end{definition}
In this paper, for simplicity, we set $\gamma = \sqrt{2}$. Let $B(z,w) = e^{z\bar{w}}$ denote the classical Fock kernel. In \cite{RBF}, the authors proved the following relation between the reproducing kernel of the RBF space and the classical Fock space:

\begin{equation}
\label{rel}
K_{\sqrt{2}} (z,w) =e^{-\frac{z^2+\bar{w}^2}{2}} B (z,w).
\end{equation}

Using the expression above, we obtain the following relation between superoscillations and the reproducing kernel of the RBF space.

\begin{lemma}
\label{rbfsuper}
Let $n \in \mathbb N$, and let $\{h_j(n)\}$ be a real-valued sequence satisfying $\sup_{j,n} |h_j(n)| \leq 1$.  
We can then express the generalized superoscillation sequence given by \eqref{Fourierg} as  
$$
f_n(z) = e^{\frac{z^2}{2}} r_n(z),
$$
where  
$$
r_n(z) = \sum_{j=0}^{n} Z_j(n,a) \, e^{-\frac{h_j^2(n)}{2}} K_{\sqrt{2}}(z, -ih_j(n)), \qquad a > 1,
$$
and $Z_j(n,a)$ are complex-valued coefficients.
\end{lemma}
\begin{proof}
The result follows by plugging \eqref{rel} into the definition of superoscillations, see \eqref{gensuper}:
	\begingroup\allowdisplaybreaks
	\begin{align*}
		f_n(z)&=\sum_{j=0}^{n} Z_j(n,a)B(z,-ih_j(n))\\
		&=\sum_{j=0}^{n} Z_j(n,a)e^{\frac{z^2-h_j^2(n)}{2}} K_{\sqrt{2}}(z,-ih_j(n))\\
		&= e^{\frac{{z}^2}{2}} \sum_{j=0}^{n} Z_j(n,a) e^{-\frac{h_j^2(n)}{2}} K_{\sqrt{2}}(z,-ih_j(n)).
	\end{align*}
	\endgroup
\end{proof}
\begin{remark}
If, in Lemma \ref{rbfsuper}, we take the sequence $h_j(n)$ to be $z_j := -i\left(1 - \frac{2j}{n}\right)$, we can establish the following relation between the superoscillations in \eqref{eq:suposc} and the reproducing kernel of the RBF space:
$$
F_n(z,a) = e^{\frac{z^2}{2}} R_n(z,a), \qquad 
R_n(z,a) = \sum_{j=0}^n C_j(n,a) \, e^{\bar{z}_j^2} K_{\sqrt{2}}(z, z_j).
$$
\end{remark}
In this section, we introduce two possible notions of superoscillations based on the RBF space. The first is derived from the relation \eqref{rel}, while in the second, the kernel $B(z, -h_j(n))$ of the Fock space, which appears in the definition of superoscillations, is replaced by the reproducing kernel of the RBF space.

\begin{definition}
Let $a>1$ and $n \in \mathbb{N}$ be fixed. For a real-valued sequence $\{h_j(n)\}$ satisfying $|h_j(n)| \leq 1$ for all $n\in\mathbb{N}$ and $j\in\mathbb{N}_0$, we define the RBF superoscillations of the first type as
\begin{equation}
\label{r1}
	r_n(z) = \sum_{j=0}^{n} Z_j(n,a) \, e^{-\frac{h_j^2(n)}{2}} K_{\sqrt{2}}(z,-i h_j(n)),
\end{equation}
and the RBF superoscillations of the second type as
\begin{equation}
\label{r2}
	s_n(z,a) = \sum_{j=0}^{n} Z_j(n,a) \, K_{\sqrt{2}}(z, -i h_j(n)),
\end{equation}
where $Z_j(n,a)$ are complex-valued coefficients.
\end{definition}

\begin{remark}
	Classical superoscillations are defined for real variables, whereas, for our purposes, RBF superoscillations are defined directly for complex variables. 
\end{remark}

\begin{example}
\label{Rbfe}
Let $z_j := -i \left(1 - \frac{2j}{n}\right)$. An example of RBF superoscillations of the first type is given by
\begin{equation}
\label{R1}
	R_n(z,a) = \sum_{j=0}^{n} C_j(n,a) \, e^{\frac{\bar{z}_j^2}{2}} K_{\sqrt{2}}(z, z_j),
\end{equation}
and an example of RBF superoscillations of the second type is
\begin{equation}
	\label{R2}
	S_n(z,a) = \sum_{j=0}^{n} C_j(n,a) \, K_{\sqrt{2}}(z, z_j),
\end{equation}
where the coefficients $C_j(n,a)$ are defined in \eqref{eq:cjzj}.
\end{example}

\begin{remark}
One of the main differences between classical superoscillations and RBF superoscillations is that the latter belong to $L^2$, whereas the former do not.
\end{remark}

The goal of this section is to establish a connection between RBF superoscillations and the complex representer theorem. In particular, we aim to provide a solution to the following problem:

\begin{problem}
\label{pb3}
	Let $n\in\mathbb{N}, (h_j(n))_{0\leq j\leq n}$ be a real-valued sequence such that $\sup_{0 \leq j \leq n}|h_j(n)| \leq 1$.
For an input set of the form $ \{z_j\}_{j=0}^n= \biggl \{-i h_j(n) \biggl\}_{j=0}^n$, we look for the corresponding output set $w_0, \dots, w_n$ such that the RBF superoscillations of the first type (see \eqref{r1}) and of the second type (see \eqref{r2}) solve the learning problem
	\begin{equation}
	\label{minn}
f_* = \arg \min_{f \in \mathcal{H}_\gamma(\mathbb{C})} J(f),
	\end{equation}
	
	where
$$
J(f) = \sum_{k=0}^n \left| w_k -f(z_k) \right|^2 + \lambda \|f\|^2, \qquad \lambda>0, \qquad f\in \mathcal{H}_\gamma(\mathbb{C}).
$$
	
\end{problem}

In the following result, we provide a solution to the above problem:

\begin{theorem}
	\label{rbf3}
The components of the output set $\{w_k\}_{k=0}^n$ in Problem \ref{pb3} are given as follows:

\begin{itemize}
	\item[i)] For the RBF superoscillations of the first type:
	\begin{equation}
		\label{s1rbf}
		w_k = e^{\frac{h_k^2(n)}{2}} \sum_{j=0}^{n} Z_j(n,a) \, e^{-h_k(n) h_j(n)} 
		+ \lambda Z_k(n,a) \, e^{-\frac{h_k^2(n)}{2}},
	\end{equation}
	
	\item[ii)] For the RBF superoscillations of the second type:
	\begin{equation}
		\label{s2rbf}
		w_k = e^{\frac{h_k^2(n)}{2}} \sum_{j=0}^{n} Z_j(n,a) \, e^{\frac{h_j^2(n)}{2} - h_k(n) h_j(n)} 
		+ \lambda Z_k(n,a) .
	\end{equation}
\end{itemize}
where $Z_j(n,a)$ are complex values coefficents.
\end{theorem}
\begin{proof}
First, we solve Problem \ref{pb3} by considering the RBF superoscillations of the first type. By the complex representer theorem (see Theorem \ref{thm:crt}), the minimizer is given by
$$
f_*(z) = \sum_{j=0}^n Z_j(n,a) \, e^{-\frac{z_j^2}{2}} K_{\sqrt{2}}(z, z_j).
$$
We take as input set $ \{z_j\}_{j=0}^n= \biggl \{-i h_j(n) \biggl\}_{j=0}^n$. By Proposition \ref{prop:w}, the corresponding output set $\{w_k\}_{k=0}^n$ for Problem \ref{pb3} is given by

\begin{equation}
\label{sol}
	\bw= (K+\lambda I) \balpha,
\end{equation}
where
\[
\boldsymbol{\alpha} = \left( Z_0(n,a) \, e^{\frac{-h_0^2(n)}{2}}, \dots, Z_n(n,a) \, e^{\frac{-h_n^2(n)}{2}} \right)^T, 
\qquad
K = \left( e^{-\frac{(h_k(n) - h_j(n))^2}{2}} \right)_{0 \leq k,j \leq n}.
\]
Now, by computing $K \boldsymbol{\alpha}$, we obtain
\begin{eqnarray}
	\nonumber
K \balpha&=& \sum_{j=0}^{n} Z_j(n,a) e^{- \frac{h_j^2(n)}{2}}K_{\sqrt{2}}(-ih_k(n),-ih_j(n))\\
\nonumber
&=& \sum_{j=0}^{n} Z_j(n,a) e^{- \frac{h_j^2(n)}{2}} e^{\frac{(h_k(n)-h_j(n))^2}{2}}\\
\label{finalre}
&=&e^{\frac{h_k^2(n)}{2}} \sum_{j=0}^{n} Z_j(n,a) \, e^{-h_k(n) h_j(n)}.
\end{eqnarray}
By plugging \eqref{finalre} in \eqref{sol} we get \eqref{s1rbf}. Next, we address Problem \ref{pb3} by considering the RBF superoscillations of the second type. In this case, the minimizer, according to the complex representer theorem (see Theorem \ref{thm:crt}), is
\[
f_*(z) = \sum_{j=0}^n Z_j(n,a) \, K_{\sqrt{2}}(z, z_j).
\]
  
  We choose the input set $ \{z_j\}_{j=0}^n= \biggl \{-i h_j(n) \biggl\}_{j=0}^n$. Then, by Proposition \ref{prop:w}, the associated output set $\{w_k\}_{k=0}^n$ for Problem \ref{pb3} is determined by

	\begin{equation}
	\label{sol2}
	\bw=(K+\lambda I) \balpha,
\end{equation}
where
$$ \balpha= \left(Z_0(n,a),..., Z_n(n,a)  \right)^T,
\quad
K= \left(e^{- \frac{(h_k(n)-h_j(n))^2}{2}}\right)_{0 \leq k, j \leq n},
$$
By evaluating the product $K \boldsymbol{\alpha}$, we then find
\begin{eqnarray}
\nonumber
K \balpha&=& \sum_{j=0}^{n}Z_j(n,a) K_{\sqrt{2}}(-ih_k(n),-ih_j(n))\\
\nonumber
&=& \sum_{j=0}^{n} Z_j(n,a) e^{\frac{(h_k(n)-h_j(n))^2}{2}}\\
\label{finalrbs}
&=& e^{\frac{h_k^2(n)}{2}} \sum_{j=0}^{n} Z_j(n,a) \, e^{\frac{h_j^2(n)}{2} - h_k(n) h_j(n)}.
\end{eqnarray}
By plugging \eqref{finalrbs} into \eqref{sol2} we get \eqref{s2rbf}.
\end{proof}

The components given in Theorem \ref{rbf3} that solve Problem \ref{pb3} can be written in closed form when specific examples of RBF superoscillations are considered, such as those presented in Example \ref{Rbfe}.

\begin{theorem}
Let $n \in \mathbb{N}$ and $a > 1$. If, in Problem \ref{pb3}, we consider the RBF superoscillations $R_n(z,a)$ (see \eqref{R1}) as the solution to the minimization problem \eqref{minn}, then the components of the set $\{w_k\}_{k=0}^n$ that solve Problem \ref{pb3} are given by
\begin{eqnarray}
\label{eq:wkrbf}
w_k &=& \left(\frac{1+a}{2}\right)^n \left[e^{-\frac{1}{2}+\frac{2k^2}{n}}  \left(1+ \left(\frac{1+a}{1-a}\right) e^{\frac{2}{n} \left(1-\frac{2k}{n}\right)}\right)^n\right.\\
\nonumber
&&+ \lambda \left.
{n\choose k}\left(\frac{1-a}{1+a}\right)^{k} e^{\left(
	-\frac{2k^2}{n^2}+\frac{2k}{n} -\frac{1}{2}
	\right) } \right].
\end{eqnarray}

\end{theorem}

 \begin{proof}
Since we are taking the superoscillations $R_n(z,a)$ defined in \eqref{R1} as the solution to the minimization problem, it suffices to apply formula \eqref{comp} with $Z_\ell(n,a) = C_\ell(n,a)$ where the coefficients $C_\ell(n,a)$ are defined in \eqref{eq:cjzj}, and $h_\ell(n) = 1 - \frac{2\ell}{n}$ for $\ell = j, k$. Hence, using \eqref{s1rbf}, we obtain
\begin{eqnarray}
\nonumber
\lambda Z_k(n,a) e^{-\frac{h_k^2(n)}{2}}&=& \lambda \binom{n}{k} \left(\frac{1+a}{2}\right)^n \left(\frac{1-a}{1+a}\right)^k e^{- \frac{1}{2} \left(1-\frac{2k}{n}\right)^2}\\
\label{ef1}
&=&\lambda \binom{n}{k} \left(\frac{1+a}{2}\right)^n \left(\frac{1-a}{1+a}\right)^k e^{-\frac{1}{2}-\frac{2k^2}{n^2}+\frac{2k}{n}},
\end{eqnarray}
and
		\begingroup\allowdisplaybreaks
\begin{eqnarray}
\nonumber
 e^{\frac{h_k^2(n)}{2}} \sum_{j=0}^{n} Z_j(n,a) e^{-h_k(n)h_j(n)}&=& e^{\frac{1}{2}+\frac{2k^2}{n}-\frac{2k}{n}} \left(\frac{1+a}{2}\right)^n \sum_{j=0}^{n} \binom{n}{k} \left(\frac{1-a}{1+a}\right)^j e^{- \left(1-\frac{2j}{n}\right) \left(1-\frac{2k}{n}\right)}\\
 \nonumber
 &=& e^{-\frac{1}{2}+\frac{2k^2}{n}} \left(\frac{1+a}{2}\right)^n \sum_{j=0}^{n} \binom{n}{k} \left(\frac{1-a}{1+a}\right)^j e^{ \frac{2j}{n}-\frac{4jk}{n^2}}\\
 \nonumber
 &=&e^{-\frac{1}{2}+\frac{2k^2}{n}} \left(\frac{1+a}{2}\right)^n \sum_{j=0}^{n} \binom{n}{k} \left(\frac{1-a}{1+a}\right)^j \left(e^{\frac{2}{n}-\frac{4k}{n^2}} \right)^j\\
 \label{ef2}
 &=& e^{-\frac{1}{2}+\frac{2k^2}{n}} \left(\frac{1+a}{2}\right)^n \left[1+ \left(\frac{1+a}{1-a}\right) e^{\frac{2}{n} \left(1-\frac{2k}{n}\right)}\right]^n.
\end{eqnarray}
\endgroup
The final result follows by combining \eqref{ef1} and \eqref{ef2}.

\end{proof}

\begin{proposition}
Let $n \in \mathbb{N}$ and $a > 1$. If, in Problem \ref{pb3}, we consider the RBF-superoscillations $S_n(z,a)$ defined in \eqref{R2} as the solution to the minimization problem \eqref{minn}, then the components of the set $\{w_k\}_{k=1}^n$ that solve Problem \ref{pb3} are given by
\begin{equation}
w_k= \left(\frac{1+a}{2}\right)^n \left[ \lambda \binom{n}{k} \left(\frac{1+a}{1-a}\right)^k +e^{\frac{2k^2}{n^2}} \sum_{j=0}^{n} \binom{n}{j} \left(\frac{1-a}{1+a}\right)^j e^{ \frac{2j\left(j-\frac{2k}{n}\right)}{n^2}}\right].
\end{equation}
\end{proposition}
\begin{proof}
As we take the superoscillations $R_n(z,a)$ defined in \eqref{R1} to be the solution of the minimization problem, formula \eqref{comp} can be applied with $Z_\ell(n,a) = C_\ell(n,a)$ and $h_\ell(n) = 1 - \frac{2\ell}{n}$ for $\ell = j, k$. Accordingly, by employing \eqref{s1rbf}, we obtain
\begingroup\allowdisplaybreaks
\begin{eqnarray*}
 e^{\frac{h_k^2(n)}{2}} \sum_{j=0}^{n} Z_j(n,a) e^{\frac{h_j^2(n)}{2}-h_k(n)h_j(n)}&=&  e^{\frac{1}{2} \left(1-\frac{2k}{n} \right)^2} \left(\frac{1+a}{2}\right)^n\sum_{j=0}^{n} \binom{n}{j} \left(\frac{1+a}{1-a}\right)^j \times\\
 && \times e^{\frac{1}{2} \left(1-\frac{2j}{n}\right)^2-\left(1-\frac{2k}{n}\right) \left(1-\frac{2j}{n}\right)} \\
 &=&  e^{\frac{2k^2}{n^2}} \left(\frac{1+a}{2}\right)^n \sum_{j=0}^{n} \binom{n}{j} \left(\frac{1-a}{1+a}\right)^j e^{ \frac{2j\left(j-\frac{2k}{n}\right)}{n^2}}
\end{eqnarray*}
\endgroup
\end{proof}

\begin{example}
	In a similar manner to Example \ref{ex:1}, the solution to the minimization problem represents the separating hyperplane of the data, in this case the RBF-Superoscillation sequence.
	
    Considering $n=100$, $a=2$, and $\lambda =1$, we observe how the two superoscillation functions $\tilde{R}_{100}(z,2)$ and $\tilde{S}_{100}(z,2)$ separates the data points $\{(z_k, w_k)\}_{k=0}^{100}$, where $z_k$ is given by \eqref{eq:cjzj}, and $w_k$ is given by \eqref{eq:wkrbf}. In fact, as a consequence of the minimization problem, the RBF superoscillation is the best function from the hypothesis class $\mathcal{H}_{\sqrt{2}}$ that separates the data points.
    In Figures \ref{rbfdata1} and \ref{rbfdata2}, we see the data points plotted by $x$-axis represents $z_k$, and $y$-axis represents the outputs $w_k$. And in the same figures, we also see the RBF-Superoscillation $\tilde{R}_{100}(z,2)$ and $\tilde{S}_{100}(z,2)$ are begin plotted along the $x$-axis separating the data points respective. 
    
    Afterward, in Figures \ref{rbfdata1log} and \ref{rbfdata2log}, we present the log-scaled version of the graphs to better illustrate the osculatory behavior of the functions.
    
	\begin{figure}[H]
		\minipage{0.5\textwidth}
		\includegraphics[width=\linewidth]{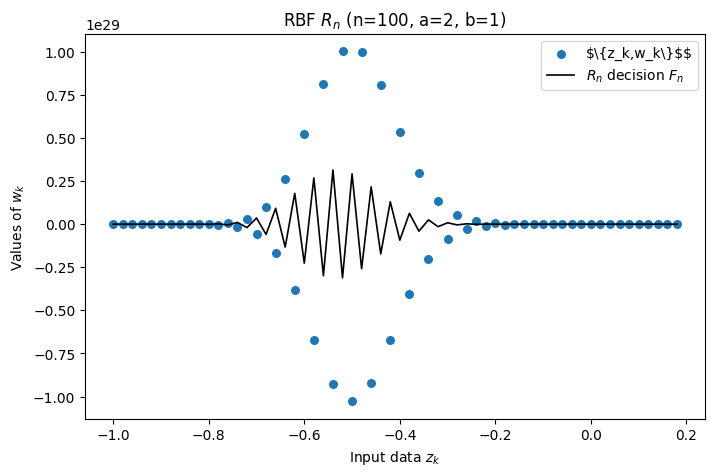}
		\caption{The RBF-Superoscillation of the first kind $R_{100}(z,2)$ ploted on the $x$-axis seperated the data points $\{(z_k,w_k)\}_{k=0}^{100}$.}
        \label{rbfdata1}
		\endminipage
		\minipage{0.5\textwidth}
		\includegraphics[width=\linewidth]{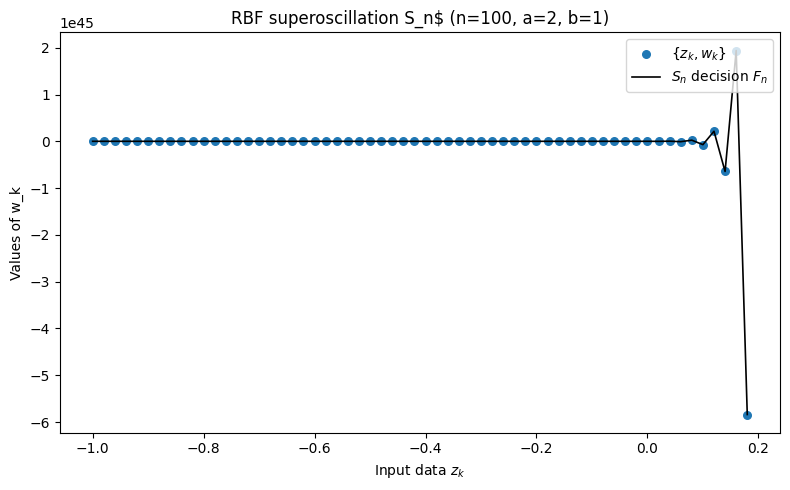}
		\caption{First 60 entries of the RBF-Superoscillation of the second kind $S_{100}(z,2)$ ploted on the $x$-axis seperated the data points $\{(z_k,w_k)\}_{k=0}^{100}$.}
        \label{rbfdata2}
		\endminipage\\
        \minipage{0.5\textwidth}
		\includegraphics[width=\linewidth]{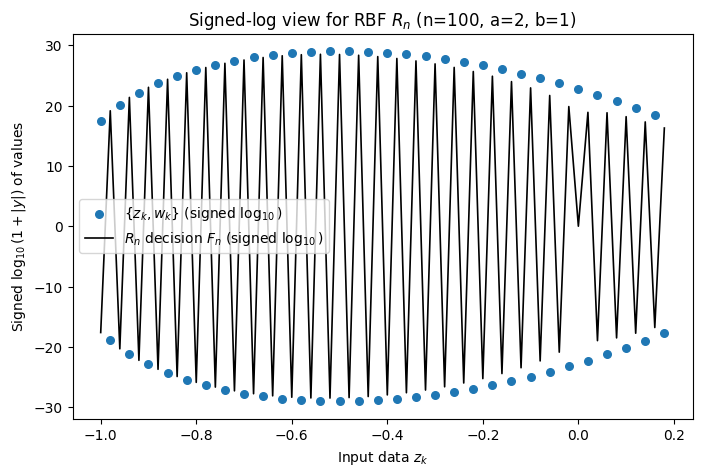}
		\caption{Log-scaled RBF-Superoscillation of the first kind $R_{100}(z,2)$ ploted on the $x$-axis seperated the data points $\{(z_k,w_k)\}_{k=0}^{100}$.}
        \label{rbfdata1log}
		\endminipage
		\minipage{0.5\textwidth}
		\includegraphics[width=\linewidth]{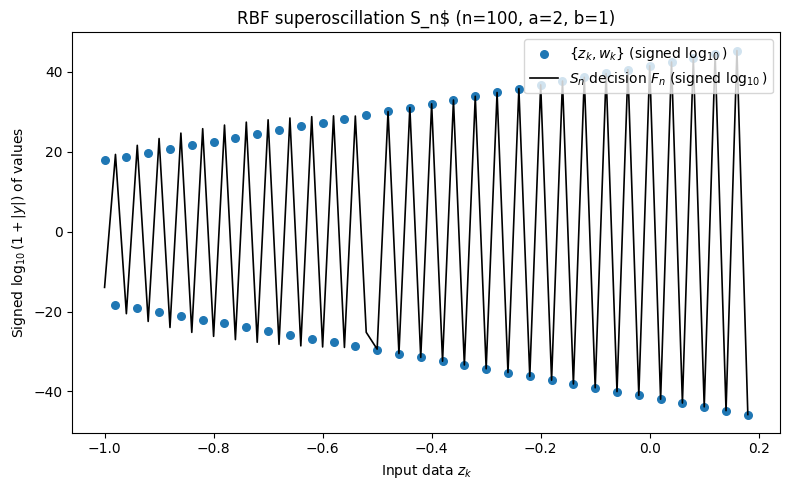}
		\caption{First 60 entries of the Log-scaled RBF-Superoscillation of the second kind $S_{100}(z,2)$ ploted on the $x$-axis seperated the data points $\{(z_k,w_k)\}_{k=0}^{100}$.}
        \label{rbfdata2log}
        \endminipage

	\end{figure}
\end{example}

\section{Supershift Property and RKHS}

In recent years, it has been shown that the concept of superoscillating functions can be extended to the more general notion of supershift.
The supershift property generalizes superoscillations, and it has become a key tool in the analysis of the time evolution of superoscillatory functions when they are used as initial data in the Schrödinger equation; see \cite{ABCS,ABCS1, ABCS2,  BCSS2023}.

\begin{definition}[Supershift]
\label{supers}
Let $ \mathcal{I} \subseteq \mathbb{R}$ be an interval with $[-1,1] \subseteq \mathcal{I}$ and let $\varphi: \mathcal{I} \times \mathbb{R} \to \mathbb{R}$, be a continuos function on $\mathcal{I}$. We set
$$ \varphi_h(x):= \varphi(h,x), \quad h \in \mathcal{I}, \, \, x \in \mathbb{R}$$
and we consider a sequence of points $(h_j(n))$ such that
$$ h_j(n) \in [-1,1], \quad \hbox{for} \quad j=0,...,n \quad \hbox{and} \quad n \in \mathbb{N}.$$
We define the functions

\begin{equation}
	\label{super1}
	\psi_n(x)= \sum_{j=0}^{n} c_j(n) \varphi_{h_j(n)}(x)
\end{equation}
where $(c_j(n))$ is a sequence of complex numbers for $j=0,...,n$ and $n \in \mathbb{N}$. If
$$ \lim_{n \to \infty} \psi_n(x)=\varphi_a(x),$$
for some $a \in \mathcal{I}$ with $|a|>1$, we say that the function $\varphi_n(x)$, for $x \in \mathbb{R}$, admits a supershift.

\end{definition}

\begin{remark}
	The term supershift arises from the fact that the interval $\mathcal{I}$ can be chosen arbitrarily large (even equal to the whole real line $\mathbb{R}$), and that the constant $a$ may lie arbitrarily far from the interval $[-1,1]$, where the functions $\varphi_{h_{j,n}(.)}$ are indexed; see \eqref{super1}.
\end{remark}

\begin{remark}
\label{rem}
We observe that if the function $\varphi_{h}(x)$ in Definition~\ref{supers} is analytic in $\lambda$ as a complex variable, then the mapping $\lambda \mapsto \varphi_h(x)$ satisfies the supershift property. This result can be regarded as a direct consequence of \cite[Thm.~4.8]{superss}, established in the case of two variables. For further discussions on the connection between the supershift property and analyticity we refer to \cite{CSSY2025CAS, CSSY2025}.
\end{remark}

\begin{remark}
The notion of supershift has also been investigated in the framework of several complex variables, see \cite{CPS}, as well as in the hypercomplex setting, see \cite{CMP}.

\end{remark}

In this section, one of our main objectives is to construct new superoscillating sequences by replacing the exponential function in the classical superoscillations—namely, the reproducing kernel of the Fock space—with reproducing kernels of certain generalized Fock spaces.

\subsection{Mittag Leffler Fock Space}

A generalized construction of the Fock space based on a Mittag–Leffler reproducing kernel was introduced in \cite{mlf}. This framework has been the topic of several research works; in particular, the creation and annihilation operators were studied in \cite{mlf}, the number states were analyzed in \cite{TN}, and the corresponding Bargmann transform was developed in various settings in \cite{ADK0, KN}.

\begin{definition}\label{def:mlf}
Let $q>0$, the Mittag Leffler Fock space (MLF space) is defined as
 $$\displaystyle ML_q(\mathbb{C})=\left\lbrace  f(z)=\sum_{n=0}^{\infty}a_nz^n,\quad \frac{1}{q\pi}\int_{\mathbb{C}}|f(z)|^2|z|^{\frac{2}{q}-2}e^{-\frac{|z|^2}{q}}dA(z)< \infty \right\rbrace.$$
\end{definition}

It is important to note that the reproducing kernel function associated to $ML_q(\mathbb{C})$ can be expressed in terms of the classical Mittag-Leffler function. In particular, this leads to the kernel 
\begin{equation}\label{eq:mlfkernel}
E_q(z,w):=\displaystyle \sum_{n=0}^\infty \frac{(z\overline{w})^n}{\Gamma(qn+1)},
\end{equation}
for every $z,w\in \mathbb{C}$.

\begin{remark}
The particular case $q=1$ corresponds to the classical Fock space.
\end{remark}

By using \eqref{eq:mlfkernel} we provide the definition of Mittag Leffler superocillations.

\begin{definition}
\label{superq}
Let $q>0$, $z \in \mathbb{C}$ and $a>1$.  We assume that $\{h_j(n)\}$ is a real-valued sequence  satisfying $|h_j(n)| \leq 1$ for all $n\in\mathbb{N}$ and $j \in \mathbb{N}_0$. Then the Mittag-Leffler superoscillations are defined as
\begin{equation}
\label{MLs}
f_{n,q}(z)=\sum_{j=0}^{n} Z_j(n,a)E_q(z, z_j), \quad z_j=-ih_j(n),
\end{equation}
where $Z_j(n,a)$ are complex-valued coefficients.
\end{definition}

Since the function $E_q(z, z_j)$ is holomorphic in the variable $z$,  for $a>1$ by the supershift property (see Remark \ref{rem}) we have

$$ \lim_{n\to\infty } F_{n,q}(z,a)=E_q(z,-ia).$$

\subsection{Touchard Polynomials and Fock Space}

In \cite{ADK}, the authors studied the images of the Schwartz space and of its dual (the space of tempered distributions) under the Segal–Bargmann transform. The characterization of these images leads to interesting reproducing kernel Hilbert spaces. One of these is given by the following space:
$$ \mathcal{H}_p(\mathbb{C}) := \left\{ f(z) = \sum_{n=0}^\infty z^n a_n, \quad (a_n)_{n \in \mathbb{N}_0} \subseteq \mathbb{C}, \quad \sum_{n=0}^\infty \frac{|a_n|^2 n!}{(n+1)^{2p}} < \infty \right\},$$

for $p \in \mathbb{N}_0$. The space $\mathcal{H}_p(\mathbb{C})$ is a reproducing kernel Hilbert space with the kernel given by
$$ K_p(z, w) =\frac{1}{z \overline{w}} T_{2p+1}(z \overline{w})e^{z\bar{w}},
$$
where $T_{n}$ are the Touchard polynomials (see \cite{T}), given by
$$ T_n(x) := \sum_{k=0}^n S(n, k) x^k, \qquad S(n, k) := \frac{1}{k!} \sum_{i=0}^k (-1)^i \binom{k}{i} (k-i)^n.$$
where $S(n, k)$ represents the Stirling numbers of the second kind. The generating function for the Touchard polynomials is given by
\begin{equation}
\label{genn}
T_n(x)= e^{-x} \sum_{k=0}^{\infty} \frac{k^n x^k}{k!}.
\end{equation}

\begin{remark}
The space $ \mathcal{H}_p(\mathbb{C})$ coincides with the Fock space if we consider $p=0$.
\end{remark}

\begin{definition}
\label{superp}
Let $p>0$ and $a>1$. We suppose that $\{h_j(n)\}$ is a real-valued sequence  satisfying $|h_j(n)| \leq 1$ for all $n\in\mathbb{N}$ and $j \in \mathbb{N}_0$. Then the Touchard superoscillations are defined as

\begin{equation}
\label{N1}
f_{n,p}(z)=\sum_{j=0}^{n} Z_j(n,a)K_p(z, z_j), \qquad z_j=-i h_j(n).
\end{equation}

\end{definition}

Since the function $K_p(z,z_j)$ is holomorphic in the variable $z$, by the supershift property we have

\[
\lim_{n\to\infty} f_{n,p}(z) = K_p(z,-ia).
\]

\begin{remark}
\label{superpp}
If we set $m = 1$ in \eqref{MLs} and $p = 0$ in \eqref{N1}, then we recover the classical notion of superoscillations, see \eqref{eq:suposc}.
\end{remark}

To establish a connection between the Touchard superoscillations and the classical superoscillations (see \eqref{gensuper}), we recall the creation and annihilation operators in the Fock space:
$$
M_z f(z):= z f(z), \qquad \partial f(z):= \frac{d}{dz} f(z).
$$

\begin{proposition}
Let $p \in \mathbb{N}_0$ and $n \in \mathbb{N}$. Then the Touchard superoscillations $f_{n,p}(z)$ are related to the classical superoscillations as follows:
\begin{equation}
\label{psuper1}
f_{n,p}(z) = \left( \mathcal{I} + M_z \partial_z \right)^{2p} f_n(z).
\end{equation}
\end{proposition}
\begin{proof}
We prove the result by induction on $p$. The result is trivial for $p=0$, see Remark \ref{superpp}. We suppose that \eqref{psuper} is valid for $p$ and we prove it for $p+1$. By using \eqref{genn} we have
\begin{equation}
K_p(z,z_j)= \frac{1}{z z_j} \sum_{k=0}^\infty \frac{k^{2p+1} z^k z_{j}^k}{k!}= \sum_{k=1}^{\infty} \frac{k^{2p} z^{k-1}z_j^{k-1}}{(k-1)!}= \sum_{k=0}^{\infty} \frac{(k+1)^{2p} z^{k}z_j^{k}}{k!}.
\end{equation}

By using the above identity and the inductive hypothesis we have
\begin{eqnarray*}
(\mathcal{I}+M_z \partial)^{2p+2} f_{n}(z)&=&(\mathcal{I}+M_z \partial)^{2} f_{n,p}(z)\\
&=& (\mathcal{I}+M_z \partial)^{2} \left( \sum_{j=0}^{n}Z_j(n,a) K_p(z,z_j)\right)\\
&=&\sum_{j=0}^{n}Z_j(n,a) K_p(z,z_j)+2 \sum_{j=0}^{n}Z_j(n,a) (M_z \partial) K_p(z,z_j)\\
&&+\sum_{j=0}^{n}Z_j(n,a) (M_z \partial)^2 K_p(z,z_j)\\
&=& \sum_{j=0}^{n}Z_j(n,a) \sum_{k=0}^{\infty} \frac{(k+1)^{2p} z^k z_j^k}{k!}+2\sum_{j=0}^{n}Z_j(n,a) \sum_{k=0}^{\infty} \frac{(k+1)^{2p} kz^k z_j^k}{k!}\\
&&+\sum_{j=0}^{n}Z_j(n,a) \sum_{k=0}^{\infty} \frac{(k+1)^{2p} k^2z^k z_j^k}{k!}\\
&=& \sum_{j=0}^n Z_j(n,a) \left[\sum_{k=0}^\infty \frac{(k+1)^{2p+2} z^k z_j^k}{k!}\right]\\
&=&\sum_{j=0}^{n}Z_j(n,a) K_{p+1}(z,z_j)\\
&=& f_{n,p+1}(z).
\end{eqnarray*}
This proves the result.
\end{proof}

\begin{remark}
	In \cite{SupBook} superoscillations were originally introduced in the setting of real variables, and later  they were extended to the complex setting. For our purposes, the Mittag–Leffler  and Touchard superoscillations have been defined directly in the complex setting.
\end{remark}

\subsection{Representer Theorem and Supershift Property}

In this part of the section, our aim is to establish a correspondence between Mittag–Leffler superoscillations and Touchard superoscillations on one hand, and the complex Representer theorem on the other. More precisely, we seek to solve the following two problems:

\begin{prob}
\label{Mlp}
Let $n\in\mathbb{N}, (h_j(n))_{0\leq j\leq n}$ be a real-valued sequence such that $\sup_{0 \leq j \leq n}|h_j(n)| \leq 1$.
For an input set of the form $ \{z_j\}_{j=0}^n= \biggl \{-i h_j(n) \biggl\}_{j=0}^n$, we seek a corresponding output set 
$\{w_0,\dots,w_n\}$ such that the Mittag--Leffler superoscillations $f_{n,q}(z)$ solves the learning problem 
$$
f_* = \underset{f \in ML_q(\mathbb{C})}{\arg\min}\, J(f), $$

	where
$$
J(f) = \sum_{k=0}^n \left| w_k -f(z_k) \right|^2 + \lambda \|f\|^2, \qquad \lambda>0, \qquad f\in ML_q(\mathbb{C}).
$$

\end{prob}

\begin{prob}
\label{Tp}
Let $n\in\mathbb{N}, (h_j(n))_{0\leq j\leq n}$ be a real-valued sequence such that $\sup_{0 \leq j \leq n}|h_j(n)| \leq 1$.
For an input set of the form $ \{z_j\}_{j=0}^n= \biggl \{-i h_j(n) \biggl\}_{j=0}^n$,  we aim to find a corresponding output set 
$\{w_0,\dots,w_n\}$ such that the Touchard superoscillations $f_{n,p}(z)$ solve the learning problem
\begin{equation}
	\label{minp}
	f_* = \underset{f \in \mathcal{H}_p(\mathbb{C})}{\arg\min}\, J(f),
\end{equation}

	where
$$
J(f) = \sum_{k=0}^n \left| w_k -f(z_k) \right|^2 + \lambda \|f\|^2, \qquad \lambda>0, \qquad f\in \mathcal{H}_p(\mathbb{C}).
$$
\end{prob}

In the following result we provide a solution of Problem \ref{Mlp}.

\begin{proposition}
	\label{qsuper}
The entries of the output vector $\{w_k\}_{k=0}^n$ that solve Problem~\ref{Mlp} are given by

\begin{equation}
	\label{one}
	w_k=\lambda Z_k(n,a)+ \sum_{j=0}^{n}Z_j(n,a)E_q(-ih_k(n), -ih_j(n)),
\end{equation}
where $h_{\ell}(n)$ for $\ell=j,k$ is a sequence such that $ |h_\ell(n)| \leq 1$ for all $\ell \in \mathbb{N}_0$.
\end{proposition}

\begin{proof}

According to the complex Representer Theorem (Theorem~\ref{thm:crt}), the minimizer $f_*$ can be written as

$$ f_*(z)= \sum_{j=0}^{n}Z_j(n,a)E_q(z_k,z_j).$$
We take the input set as $z_j=\{-ih_j(n)\}_{j=0}^n$. Thus, by Proposition~\ref{prop:w}, the components of the vector $\mathbf{w}$ that solve Problem~\ref{Mlp} are given by

	\begin{equation}
		\label{s1}
		w_k=\lambda \alpha_k+K\alpha_k,
	\end{equation}
	where 
$$
		\balpha= \left(C_0(n,a),..., C_n(n,a)\right)^T, \qquad 		K= \left( \sum_{\ell=0}^\infty \frac{(h_j(n)h_k(n))^\ell}{\Gamma(q\ell+1)}\right)_{0 \leq k, j \leq n}.
$$
Since 
$$ K \balpha=\sum_{j=0}^{n}Z_j(n,a)E_q(-ih_k(n), -ih_j(n)),$$
by \eqref{s1} we get the result.
\end{proof}

By using similar arguments, we obtain the following result, which can be regarded as the solution to Problem~\ref{Tp}.

\begin{proposition}
\label{psuper}
	The components of the output vector $\{w_k\}_{k=0}^n$ in Problem~\ref{Tp} are given by
\begin{equation}
\label{pcomponent}
w_k=\lambda Z_k(n,a)+\sum_{j=0}^{n}Z_j(n,a)K_p(-ih_k(n),-ih_j(n))
\end{equation}
\end{proposition}

By fixing a specific value of $p$, and choosing suitable coefficients $Z_j(n,a)$ and sequence $h_j(n)$,  we can derive a closed-form expression for \eqref{pcomponent}. We study the particular case when $p=1$.

\begin{proposition}
\label{pcase}
Let $n \in \mathbb{N}$ and $a > 1$. If, in Problem~\ref{Tp}, we take the superoscillating sequence \eqref{eq:suposc} to be the solution of the minimization Problem \ref{Tp}, then the components of the vector $\mathbf{w} = \{w_k\}_{k=0}^n$ that solve Problem~\ref{Tp} are given by
\begin{eqnarray}
	\nonumber
w_k&=&
 \lambda \binom{n}{k} \left(\frac{1+a}{2}\right)^n \left(\frac{1-a}{1+a}\right)^k
+e^{z_k} \left(\frac{1+a}{2}\right)^n\left(1+ \left(\frac{1-a}{1+a}\right) e^{-\frac{2z_k}{n}}\right)^{n-2}\\
\nonumber
&& \biggl \{ \left[1+\left(\frac{1-a}{1+a}\right) e^{-\frac{2z_k}{n}}\right]^2+3 z_k \left[1- \left(\frac{1-a}{1+a}\right)^2 e^{-\frac{4 z_k}{n}}\right]  +\\
\label{finalp}
&&+ z_k^2 \left[ \left[1- \left(\frac{1-a}{1+a}\right) e^{-\frac{2 z_k}{n}}\right]+\frac{4}{n} \left(\frac{1-a}{1+a}\right) e^{-\frac{2 z_k}{n}}\right] \biggl \},
\end{eqnarray}
where $z_{k}:=1-\frac{2k}{n}$.
\end{proposition}
\begin{proof}
	
Since we are treating the superoscillations $F_n(z,a)$ defined in \eqref{eq:suposc} as the solution of the minimization problem, it is sufficient to apply formula \eqref{pcomponent} with $Z_\ell(n,a) = C_\ell(n,a)$ and $h_\ell(n):=z_\ell = 1 - \frac{2\ell}{n}$ for $\ell = j,k$. We set $z_j=1- \frac{2 j}{n}$. Since $T_3(z_kz_j)=(z_k z_j)+3(z_k z_j)^2+(z_k z_j)^3$  by formula \eqref{pcomponent} we obtain

\begin{eqnarray}
	\nonumber
	\sum_{j=0}^{n} C_j(n,a) K_1(z_k, z_j)
	&=& \sum_{j=0}^{n} C_j(n,a) \frac{1}{z_k z_j}T_3(z_k z_j) e^{z_kz_j}\\
	\label{summattion}
	&=& \sum_{j=0}^n C_j(n,a)e^{z_k z_j}+3\sum_{j=0}^n C_j(n,a) (z_k z_j) e^{z_k z_j}+\sum_{j=0}^n C_j(n,a) (z_kz_j)^2 e^{z_k z_j}.
\end{eqnarray}

We focus on finding the closed forms of the above summations. By the binomial theorem we deuce that
\begin{equation}
	\label{zero}
	\sum_{j=0}^n C_j(n,a)e^{z_k z_j}= e^{z_k} \left(\frac{1+a}{2}\right)^n \left[1+ \left(\frac{1-a}{1+a}\right) e^{-\frac{2}{n}z_k}\right]^n.
\end{equation}
For the other two summations we use the following facts:
\begin{equation}
	\label{newt1}
	\sum_{j=0}^n \binom{n}{j}j x^j=nx (1+x)^{n-1}, \qquad 	\sum_{j=0}^n \binom{n}{j} j^2 x^j=nx(1+x)^{n-2} (1+nx),
\end{equation}
with $ x= \left(\frac{1-a}{1+a}\right) e^{-\frac{2 z_k}{n}}$ we get
\begin{eqnarray}
\nonumber
3\sum_{j=0}^n C_j(n,a) (z_k z_j) e^{z_k z_j}&=&3z_ke^{z_k}\left(\frac{1+a}{2}\right)^n  \left[1+ \left(\frac{1-a}{1+a}\right) e^{-\frac{2}{n}z_k}\right]^{n-1}\\
\label{p2p}
 &&\left[1-\left(\frac{1-a}{1+a}\right) e^{-\frac{2}{n}z_k}\right].
\end{eqnarray}
and
\begin{eqnarray}
	\nonumber
	\sum_{j=0}^n C_j(n,a) (z_kz_j)^2 e^{z_k z_j}&=& e^{z_k} z_k^2 \left(\frac{1+a}{2}\right)^n\left(1+ \left(\frac{1-a}{1+a}\right) e^{-\frac{2z_k}{n}}\right)^{n-2} \times\\
	\label{p3}
	&& \times \biggl \{\left[1- \left(\frac{1-a}{1+a}\right) e^{-\frac{2}{n}z_k}\right]^2+\frac{4}{n}\left(\frac{1-a}{1+a}\right) e^{-\frac{2z_k}{n}} \biggl \}.
\end{eqnarray}

By plugging \eqref{zero}, \eqref{p2p} and \eqref{p3} into \eqref{pcomponent} and using the fact that $ \lambda Z_k(n,a)=\lambda  C_k(n,a)$ we get \eqref{finalp}.
\end{proof}

\begin{example}
To show how the Touchard polynomials solve Problem \ref{Tp} when $p=1$, we present the following figure.
For $n=100,a=2,$ and $\lambda =1$, we calculate the desired outputs $\{w_k\}$ and plot $\{z_k,w_k\}$, and then we plot the Touchard supershift.
    \begin{figure}[H]
        \centering
        \includegraphics[width=0.7\linewidth]{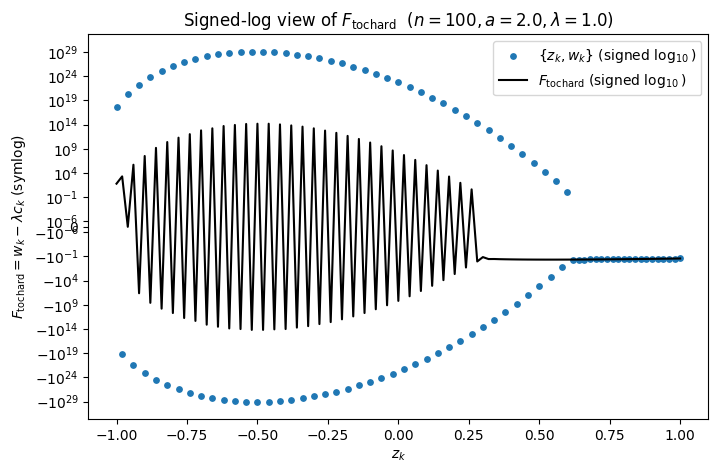}
        \label{fig:placeholder}
    \end{figure}
\end{example}

\begin{remark}
From Proposition \ref{pcase}, it becomes evident that obtaining a closed formula for \eqref{pcomponent} for arbitrary $p \in \mathbb{N}_0$ is highly nontrivial.
\end{remark}

\begin{remark}
By setting $q=1$ and $p=0$ in Propositions~\ref{qsuper} and \ref{psuper}, we recover the result proved in Theorem~\ref{class}.

\end{remark}

\section{Representer Theorem: The Hardy Space and Blaschke Product}
In this section, we present an application of the complex representer theorem by considering the Hardy space over the unit disk as the RKHS and the Blaschke product as the function to be minimized. We begin by recalling the definition of the Hardy space over the unit disk $\mathbb{D}$.

\begin{definition}[Hardy space]\label{def:hardy}
	The	Hardy space space $\mathbf H^2(\mathbb{D})$ consists of holomorphic functions $f$ over the disk such that
	$$\lim_{r\rightarrow 1}\frac{1}{2\pi}\int_0^{2\pi}|f(re^{it})|^2dt<\infty.$$
\end{definition}
The Hardy space of the unit disk is a reproducing kernel Hilbert space, whose reproducing kernel is given by
\begin{equation}
	\label{kernel1}
	K(z,w)=\displaystyle \dfrac{1}{1-z\overline{\w}} = \sum_{n=0}^{\infty} z^n\overline{\w}^n,\quad \forall z,w\in \mathbb{D},
\end{equation}
see \cite{DD} for more information on the Hardy space.
An important tool in the study of the Hardy space of the unit disk is provided by the Blaschke factors and the associated Blaschke products.  These objects play a fundamental role in the analysis of invariant subspaces and interpolation problems; see \cite{DD, GMW, R}.

\begin{definition}
	Let $z \in \mathbb{C}$, $a \in \mathbb{D}$. The Blaschke factor is defined as
	$$\frac{z - a}{1 - \bar{a} z}$$
\end{definition}
We observe that Blaschke factors define the automorphisms of $ \mathbb{D} $ up to a unimodular constant.
\begin{definition}
	Let $ \{a_j\} \subseteq \mathbb{D}$, $j=1$,...,$n$. Then Blaschke product is defined as
	\begin{equation}
		B_n(z) =  \prod_{j=1}^n \frac{z - a_j}{1 - \overline{a_j} z},
	\end{equation}
\end{definition}
\begin{remark}
	Since $|B_n(z)| = 1$, the Blaschke product is one of the simplest $\mathbf{H}^2$-inner functions. Its zeros are given by the set $\{a_j\}_{j=1,\dots,n}$.
	
\end{remark}

The derivative of $B_n(z)$ is given by 

\[
B_n'(z) = \sum_{l=1}^n \frac{1 - |a_l|^2}{(1 - \overline{a_l} z)^2} \left( \prod_{\substack{k=1 \, \,k \neq l}}^n \frac{z - a_k}{1 - \overline{a_k} z} \right),
\]
see \cite{GMW}, this implies that

\[
B_n'(a_j) = \frac{1}{1 - |a_j|^2} \prod_{\substack{k=1 \, \, k \neq j}}^n \frac{a_j - a_k}{1 - \overline{a_k} a_j}.
\]
Assuming the Blaschke product has simple zeros distinct from the origin, it can be expressed as a linear combination of the reproducing kernels of the Hardy space over the disk, namely

\begin{equation}
	\label{Bed}
	B_n(z) =c_0 + \sum_{j=1}^{n}c_j K(z,a_j), \qquad c_0=\frac{1}{\overline{B_n(0)}}, \quad c_j =\frac{1}{\bar a_j \overline{B'_n(a_j)}},
\end{equation}

see \cite{FL}, where the above identity was proved in a more general setting.

\begin{remark}
	In \cite{CQK}, the authors used the identity in \eqref{Bed} to prove the so-called Bedrosian identity. 
\end{remark}

Using \eqref{Bed}, we solve the following learning problem in the Hardy space.

\begin{problem}
	\label{HardyP}
	For an input set $z_1, \dots, z_n$ in the unit disk, we seek a corresponding output set $w_1, \dots, w_n$ such that the Blaschke product $B_n(z)$ solves the learning problem
	\[
	f_* = \underset{f \in \mathbf{H}^2(\mathbb{D})}{\arg\min} \, J(f),
	\]

		where
$$
J(f) = \sum_{k=1}^n \left| w_k -f(z_k) \right|^2 + \lambda \|f\|^2, \qquad \lambda>0, \qquad f\in \mathbf{H}^2(\mathbb D).
$$
	
\end{problem}

\begin{theorem}
	The components of the output set $ \{w_k\}_{1 \leq k \leq n}$ of the Problem \ref{HardyP} are given by 
	\begin{equation}\label{eq:bl_wk}
		w_k =  B_n(a_k) + \lambda \frac{1}{\bar a_k  \overline{B'_n(a_k)}} - \frac{1}{\overline{B_n(0)}}, \qquad 1 \leq k \leq n.
	\end{equation}
\end{theorem}
\begin{proof}
	By the complex representer theorem, see Theorem \ref{thm:crt},  the minimizer $f_*$ is given by 
	\begin{equation}
		\label{star}
		f_* (z) = \sum_{j=1}^{n} c_j K(z, z_j).
	\end{equation}
	We take as input set $\{z_j\}_{j=1}^{n} = \{a_j\}_{j=1}^{n}$, i.e., the roots of the Blaschke product. By Proposition \ref{prop:w} we know that  the output set $ \{w_k\}_{1 \leq k \leq n}$ of the Problem \ref{HardyP} are given by
	\begin{equation}
		\label{sh}
		\bw = (K+\lambda I)\balpha,
	\end{equation}
	where
	$$ \balpha = (c_1(n,a),\dots, c_n(n,a))^T, \qquad K=\left(K(a_k,a_j)\right)_{0 \leq k, j \leq n}=\left(\frac{1}{1-a_k \overline{a_j}}\right)_{1 \leq k, j \leq n}.$$
	Thus by the identity \eqref{Bed}, we have
	
	$$K\balpha =\sum_{j=1}^{n} c_j K(a_k,a_j)=\sum_{j=1}^{n} \frac{1}{\bar a_j \overline{B'_n(a_j)}} \frac{1}{1-a_k\bar{a_j}}= B_n(a_k) - \frac{1}{\overline{B_n(0)}}.$$
	So by \eqref{sh} we have
	$$
	w_k = K\balpha + \lambda \alpha_k = B_n(a_k) + \lambda \frac{1}{\bar a_k \overline{B'_n(a_k)}} - \frac{1}{\overline{B_n(0)}}, \qquad 1 \leq k \leq n.
	$$
	
\end{proof}

\begin{example}
The solution to the minimization Problem \ref{HardyP} represents the separating hyperplane of the data; in this case, it is the finite Blaschke product. 
However, since we can easily plot only a pair $(z_k, w_k)$ for real values, we must consider two cases: (1) $a_k$ real (Figure \ref{fig:awesome_image1}), or (2) $a_k$ purely imaginary  (Figure \ref{fig:awesome_image2}), as inputs to calculate the real outputs $w_k$ via \eqref{eq:bl_wk}. 

For $n = 10$ and $\lambda = 1$, we visualize how the finite Blaschke product $B_n(z)$, with randomly generated roots $a_1, \dots, a_{10}$, separates datasets of the form $\{(a_k, w_k)\}_{k=1}^{100}$, where the $w_k$ are calculated using \eqref{eq:bl_wk}.
We plot the Blaschke-based separating function
$
f_*(z) \;=\; B(z)_{10} \;-\; \frac{1}{B(0; z_k)}
$
evaluated at the sample points $z_k = a_k$ for $k=1,\cdots n$ .

\begin{figure}[H]
\minipage{0.5\textwidth}
\includegraphics[width=0.9\linewidth]{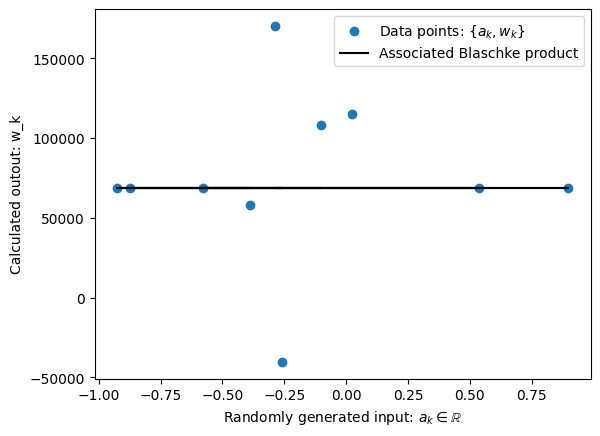}
\caption{Case (1) The Blaschke product separating ten randomly generated data points, in which the $ x$-coordinate is a randomly generated real number in $[-1,1]$, and the $y$-coordinate is calculated via \eqref{eq:bl_wk}. }\label{fig:awesome_image1}
\endminipage
\minipage{0.5\textwidth}
\includegraphics[width=0.9\linewidth]{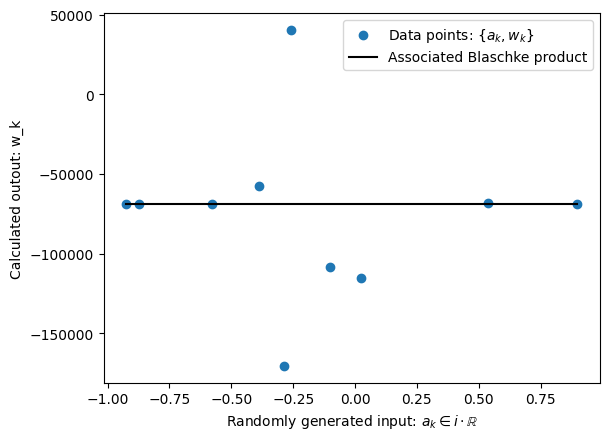}
\caption{Case (20 The Blaschke product separating ten randomly generated data points, in which the $ x$-coordinate is randomly generated as a purely imaginary number in $[-i,i]$, and the $y$-coordinate is calculated via \eqref{eq:bl_wk}.}\label{fig:awesome_image2}
\endminipage
\end{figure}
\end{example}

\section{Schrödinger Equation For The Free Particle: The RBF Superoscillations}

The time evolution of superoscillations has been extensively studied as initial conditions in the time-dependent Schrödinger equation. The first such study focused on the free particle problem, as in \cite{ACSST}. Subsequently, evolution problems with non-constant potentials were analysed in depth, including the quantum harmonic oscillator \cite{ACSS, BCSS, CSSY}, an electric field \cite{ACSSTmp}, and a uniform magnetic field \cite{CGS}. The evolution of superoscillations has also been considered for the Schrödinger equation describing spinning particles of arbitrary spin subjected to a magnetic field \cite{CPSW}. In \cite{CPSW2}, the authors investigated evolution problems for the Aharonov-Bohm potential. Superoscillation evolution have further been explored in connection with the Klein-Gordon equation \cite{ACSST1, DV}. Finally, a general framework for the evolution of superoscillations was provided in \cite{ABCS}. 
\newline
\newline
In this section, we study the free particle problem by taking as initial datum the RBF superoscillations of the first type restricted to the real line. Our goal is to determine the solution of the following Cauchy problem:

\begin{equation}
	\label{star3}
	i\frac{\partial \psi(x,t)}{\partial t}=-\frac{\partial^2\psi(x,t)}{\partial x^2}\psi(x,t), \quad \psi(x,0)=r_n(x)=e^{-\frac{x^2}{2}} f_n(x),
\end{equation}
see \eqref{r1} for the definition of $r_n(x)$, and \eqref{gensuper} for a definition of $f_n(x)$.

\begin{theorem}
	\label{Fourier1}
	Let $n \in \mathbb{N}$, $a>1$ and let $\{h_j(n)\}_{0\leq j\leq n}$ be a real-valued sequence satisfying $ |h_j(n)| \leq 1$ for all $j \in \mathbb{N}_0$. Then the solution of \eqref{star3} is given by
	
	\begin{equation}
		\psi_n(x,t)=\frac{e^{-\frac{x^2}{2(1+2it)}}}{(1+2it)^{1/2}}\sum_{j=0}^{n}Z_j(n,a) e^{-\frac{h_j^2(n)}{2}+\frac{h_j^2(n)}{2(1+2it)}+ix\frac{h_j(n)}{1+2it}}.
	\end{equation}
    where $Z_j(n,a)$ are complex values coefficients.
\end{theorem}

\begin{proof}
	In order to solve \eqref{star3} we apply the following Fourier transform 
	
	$$ \widehat{f}(\omega)=\mathcal{F}(f)(\omega)= \int_{\mathbb{R}} e^{-it \omega} f(t)dt, \quad f \in L^1(\mathbb{R}) \cap L^2(\mathbb{R}), \quad \omega \in \mathbb{R},$$
	
	to \eqref{star3}, and we get 
	\[ 
	i\frac{\partial}{\partial t } \widehat{\psi}(\lambda,t) = \lambda^2 \widehat{\psi}(\lambda,t). 
	\]
	It follows that 
	\begin{equation}
		\label{six}
		\widehat{\psi}(\lambda,t) =  c(\lambda) e^{-i\lambda^2 t} 
	\end{equation}
	where the constant $c(\lambda)$ can be calculated using the initial condition so that $\hat{\psi}(\lambda,0) = c(\lambda)$ for all $\lambda \in \mathbb{R}$. Thus, by the definition of Fourier transform we have
	
	\begin{align}
		\nonumber
		c(\lambda) & = \int_{\mathbb{R}} e^{-i\lambda x}\psi(x,0)dx\\
		\nonumber
		&=\sum_{j=0}^nZ_j(n,a) \int_{\mathbb{R}} e^{-\frac{x^2}{2} +ix\left( h_j(n)-\lambda  \right) }dx\\
		\nonumber
		& = \sum_{j=0}^nZ_j(n,a) \mathcal{F}\left[e^{-\frac{x^2}{2}}\right] \left(\lambda -h_j(n)\right) \\
		\nonumber
		&= \sqrt{2\pi}\sum_{j=0}^{n} Z_j(n,a) e^{- \frac{ (\lambda - h_j(n) )^2}{2}} \\
		\label{new1}
		& = \sqrt{2\pi}e^{-\frac{\lambda^2}{2}} \sum_{j=0}^{n}C_j(n,a) e^{\frac{h_j^2(n)}{2}}e^{\lambda h_j(n)}.
	\end{align}
	
	Hence by replacing \eqref{new1} into \eqref{six} we have
	\[
	\widehat{\psi_n}(\lambda,t) 
	= 
	\sqrt{2\pi}e^{-\frac{\lambda ^2}{2} -i\lambda^2t }  \sum_{j=0}^{n} C_j(n,a) e^{-\frac{h_j^2(n)}{2}}e^{\lambda h_j(n)}.
	\]
	So using the inverse Fourier transform we have 
	
	\begin{align}
		\nonumber
		\psi_n(x,t) & = \frac{1}{2\pi} \int_{\mathbb{R}} e^{i\lambda x} \widehat{\psi_n}(\lambda , t ) d\lambda \\ \nonumber
		&=\frac{1}{\sqrt{2\pi}} \int_{\mathbb{R}} e^{i\lambda x}
		\left(
		e^{-\frac{\lambda ^2}{2} -i\lambda^2t }  \sum_{j=0}^{n} Z_j(n,a) e^{-\frac{h_j^2(n)}{2}}e^{\lambda h_j(n)}
		\right) d\lambda\\
		\nonumber
		&=\frac{1}{\sqrt{2\pi}} \sum_{j=0}^{n}Z_j(n,a) e^{-\frac{h_j^2(n)}{2}}\int_{\mathbb{R}}e^{-\frac{\lambda^2}{2}-i\lambda^2 t +i\lambda x + \lambda h_j(n)}d\lambda\\
		\label{final}
		&=\frac{1}{\sqrt{2\pi}} \sum_{j=0}^{n}Z_j(n,a) e^{-\frac{h_j^2(n)}{2}}
		\int_{\mathbb{R}}e^{ -\lambda^2 \left(\frac{1}{2} + it\right) +\lambda \left(ix +h_j(n)\right) } d\lambda.
	\end{align}

	Recall the following integral formula for $\text{Re}(\alpha^2)>0,\beta\in\mathbb{C}$ as
\begin{equation}
\label{ref}
\int_{\mathbb{R}} e^{-\alpha^2 x^2 +\beta x} dx  = e^{\frac{\beta^2}{4\alpha^2}} \frac{\sqrt{\pi}}{\alpha},
\end{equation}
	see \cite[formula 3.323]{GR}. By taking $\alpha^2=\frac{1}{2}+it$ and $\beta = ix + h_j(n)$, we have
	\begin{equation}
		\label{intG}
		\int_{\mathbb{R}}e^{ -\lambda^2 \left(\frac{1}{2} + it\right) +\lambda \left(ix +\omega_j\right) } d\lambda
		=
		\frac{\sqrt{\pi} e^{ \frac{\left(ix+h_j(n)\right)^2 }{2+4it} } }{ \left(\frac{1}{2}+it\right)^{\frac{1}{2}}  }
		=
		\frac{\sqrt{\pi}}{\left(\frac{1}{2}+it\right)^{\frac{1}{2}}} e^{-\frac{x^2 }{2(1+2it)}}
		e^{\frac{h_j^2(n) + 2ixh_j(n) }{2(1+2it)}}
	\end{equation}
	
	By plugging \eqref{intG} into \eqref{final} we get the result.
\end{proof}

\begin{remark}
	The Cauchy problem \eqref{star3} has also been investigated in \cite{ADDS2024}, where the authors employed the notion of the short-time Fourier transform to solve the problem.
\end{remark}

\begin{remark}
	From the identities
	$$
	\left| e^{ -\frac{x^2}{2(1+2it)} } \right| = e^{-\frac{x^2}{2(1+4t^2)}}, 
	\qquad
	\left| \frac{1}{(1+2it)^{\frac{1}{2}}} \right| = \frac{1}{\sqrt{\sqrt{ 1+4t^2 }}},
	$$
	we deduce that
	$$
	|\psi_n(x,t)| \leq c(t)\, e^{-\frac{x^2}{2(1+4t^2)}},
	$$
	for some function $c(t)>0$. Hence,
	$$
	\int_{\mathbb{R}} |\psi_n(x,t)|^2 \, dx < \infty,
	$$
	which shows that $\psi_n\in L^2(\mathbb{R})$.
\end{remark}

An interesting generalization of the Cauchy problem \eqref{star3} can be obtained by considering a more general initial datum of the form
$$
\mathbf{h}_k(x) f_n(x), \qquad x \in \mathbb{R}, 
$$
where $ \mathbf{h}_k(x)$ denotes the $k$-th Hermite function, defined by
\begin{equation}
\label{hermit}
\mathbf{h}_k(x) = e^{-\frac{x^2}{2}} H_k(x), \qquad H_k(x) = k! \sum_{m=0}^{\lfloor \frac{k}{2} \rfloor} \frac{(-1)^m}{m! (k-2m)!} (2x)^{k-2m}.
\end{equation}
and $H_k(x)$ are the Hermite polynomials of degree $k$. More precisely, we aim to study the following Cauchy problem:
\begin{equation}
	\label{star4}
	i \frac{\partial \phi_{k,n}(x,t)}{\partial t} = -\frac{\partial^2 \phi_{k,n}(x,t)}{\partial x^2}, 
	\qquad \phi_{k,n}(x,0) = \mathbf{h}_k(x) F_n(x,a).
\end{equation}

\begin{remark}
The Cauchy problem \eqref{star4} generalizes \eqref{star3}. Indeed, for $k=0$ we have $h_0(x) = e^{-x^2/2}$, and therefore the initial datum in \eqref{star4} reduces to that of \eqref{star3}.
\end{remark}

\begin{remark}
	The Hermite polynomials defined in \eqref{hermit} can be extended to the holomorphic setting by replacing the real variable $x$ with the complex variable $z$.
     Moreover the following properties of the Hermite polynomials are still valid for $z$, $w \in \mathbb{C}$:
	\begin{equation}
	\label{Taylor}
	H_k(z+w)= \sum_{\ell=0}^k \binom{k}{\ell} H_k(z) (2w)^{k-\ell},
	\end{equation}
	\begin{equation}
	\label{recc}
	H_k(\gamma z)= \gamma^k\sum_{\ell=0}^{\lfloor \frac{k}{2}\rfloor} \left(\frac{\gamma^2-1}{\gamma^2}\right)^\ell \binom{k}{2 \ell} \frac{(2\ell)!}{\ell!} H_{k-2\ell}(z), \qquad \gamma \in \mathbb{C}.
	\end{equation}
\end{remark}

Before addressing the solution of \eqref{star4}, we need to establish an important preliminary result. We recall that the Hermite polynomials have the following integral representation

\begin{equation}
	\label{hermite}
	\frac{H_n(ix)}{(2i)^n} = \frac{1}{\sqrt{\pi}} \int_{-\infty}^{\infty} e^{-(y-x)^2} y^n \, dy, , \qquad n \in \mathbb{N}_0,
\end{equation}

see \cite[Theorem 4.6.6]{Mourad}. The above property of the Hermite polynomials is useful to show the following result.

\begin{lemma}
	\label{inteH}
	Let $a$, $b \in \mathbb{C} \setminus \{0\}$ with $\hbox{Re}(a)>0$, then for $n \in \mathbb{N}_0$ we have
\begin{equation}
\label{inte1}
	\int_{-\infty}^{\infty} x^n e^{-ax^2+bx} dx = \sqrt{\frac{\pi}{a}} e^{\frac{b^2}{4a}}(-i)^n   \left(\frac{1}{2\sqrt{a}}\right)^n  H_n \left(\frac{ib}{2\sqrt{a}}\right). 
\end{equation}
Moreover, if $a \neq 1$ we have
\begin{equation}
\label{inte2}
\int_{\mathbb{R}} e^{-ax^2+bx}H_k(x-c)dx=i^k \sqrt{\frac{\pi}{a}} e^{\frac{b^2}{4a}} \left(\frac{1-a}{a}\right)^{\frac{k}{2}} H_k \left(\frac{2iac-ib}{2 \sqrt{a}\sqrt{1-a}}\right)     \qquad c>0.
\end{equation}
\end{lemma}

\begin{proof}
	We set $t = \frac{b}{2\sqrt{a}}$, so we have
	\begin{equation}\label{rewrite}
		-ax^2 +bx = -a \left(  x-\frac{t}{\sqrt{a}}  \right)^2 +t^2.
	\end{equation}
	This implies that 
	$$
	\int_{-\infty}^{\infty} x^n e^{-ax^2+bx}dx = \int_{-\infty}^{\infty}x^n e^{t^2}e^{-a \left( x-\frac{t}{\sqrt{a}} \right)^2 }dx
	$$
	By making the change of variable $y = \sqrt{a} \left( x - \frac{t}{\sqrt{a}} \right)$ and by using \eqref{hermite}  we obtain
	\begin{align}
		\nonumber
		\int_{-\infty}^{\infty} e^{t^2}\left(\frac{y+t}{\sqrt{a}} \right)^n e^{-y^2} \frac{1}{\sqrt{a}} dy
		& = e^{t^2}\int_{-\infty}^{\infty} (y+t)^n e^{-y^2} a^{-\frac{n+1}{2}} dy\\
		\nonumber
		\nonumber
		&=  e^{t^2} a^{-\frac{n+1}{2}}\int_{-\infty}^{\infty} x^n e^{-(x-t)^2} dx\\
		\label{inte}
		&=  e^{t^2} a^{-\frac{n+1}{2}} \frac{\sqrt{\pi}}{2^n} (-i)^n H_n(it).
	\end{align}
Inserting in \eqref{inte} the value of $t$ we get the final result.
\\Now, we prove \eqref{inte2}. By definition of the Hermite polynomials, the binomial theorem and \eqref{inte1} we have
\begin{eqnarray*}
\int_{\mathbb{R}} e^{-ax^2+bx}H_k(x-c)dx&=&k! \sum_{m=0}^{\lfloor \frac{k}{2} \rfloor} \frac{(-1)^m 2^{k-2m}}{m!(k-2m)!} \int_{\mathbb{R}} e^{-ax^2+bx}(x-c)^{k-2m} dx\\
&=&k! \sum_{m=0}^{\lfloor \frac{k}{2} \rfloor} \frac{(-1)^m 2^{k-2m}}{m!(k-2m)!} \sum_{\ell=0}^{k-2m} \binom{k-2m}{\ell} (-c)^{k-2m-\ell} \int_{\mathbb{R}} e^{-ax^2+bx}x^{\ell} dx\\
&=& k! \sqrt{\frac{\pi}{a}} e^{\frac{b^2}{4a}} \sum_{m=0}^{\lfloor \frac{k}{2} \rfloor} \frac{(-1)^m 2^{k-2m}}{m!(k-2m)!} \sum_{\ell=0}^{k-2m} \binom{k-2m}{\ell} \frac{(-i)^\ell}{(2 \sqrt{a})^\ell}(-c)^{k-2m-\ell} H_\ell \left(\frac{ib}{2 \sqrt{a}}\right).
\end{eqnarray*}
By making some algebraic arguments we get the following:
\begin{eqnarray*}
&&\int_{\mathbb{R}} e^{-ax^2+bx}H_k(x-c)dx= i^k k! \sqrt{\frac{\pi}{a}} e^{\frac{b^2}{4a}} \sum_{m=0}^{\lfloor \frac{k}{2} \rfloor} \frac{2^{k-2m}}{m!(k-2m)!} \frac{1}{(2\sqrt{a})^{k-2m}} \times\\
&& \qquad \qquad \qquad \qquad \qquad \qquad \times \sum_{\ell=0}^{k-2m} \binom{k-2m}{\ell}(2ic\sqrt{a})^{k-2m-2\ell} H_{\ell} \left(-\frac{i b}{2 \sqrt{a}}\right).
\end{eqnarray*}
Now, by \eqref{Taylor} we have
$$\int_{\mathbb{R}} e^{-ax^2+bx}H_k(x-c)dx= i^k \sqrt{\frac{\pi}{a}} e^{\frac{b^2}{4a}} \sum_{m=0}^{\lfloor \frac{k}{2} \rfloor} \frac{k!}{m!(k-2m)!} \frac{1}{(\sqrt{a})^{k-2m}} H_{k-2m} \left(\frac{2iac-ib}{2 \sqrt{a}}\right).$$
Finally by using \eqref{recc}, with $\gamma=\sqrt{\frac{1}{1-a}}$ we have
$$\int_{\mathbb{R}} e^{-ax^2+bx}H_k(x-c)dx=i^k \sqrt{\frac{\pi}{a}} e^{\frac{b^2}{4a}} \left(\frac{1-a}{a}\right)^{\frac{k}{2}} H_k \left(\frac{2iac-ib}{2 \sqrt{a}\sqrt{1-a}}\right).$$
This proves the result.

\end{proof}

\begin{remark}
By taking $n=0$ in \eqref{inte1}, we recover the integral used in \eqref{ref}.
\end{remark}

Now, we have all the tools to solve the problem \eqref{star4}.

\begin{theorem}
	
Let  $n \in \mathbb{N}$, $k\in \mathbb{N}_0$, and $\{h_j(n)\}$ a real-valued sequence satisfying $|h_j(n)| \leq 1$ for all $0 \leq j \leq n$. The solution of the  the Cauchy problem \eqref{star4} is given by
	
	$$
	\phi_{k,n}(x,t)= \frac{ (1-2it)^{k}}{5^{\frac{k}{2}} \sqrt{1+2it}} e^{- \frac{x^2}{2(1+2it)}}\sum_{j=0}^{n} Z_j(n,a) e^{-\frac{h_j^2(n)}{2} +\frac{h_j^2(n)}{2(1+2it)}+\frac{ix h_j(n)}{1+2it}} H_k \left(\frac{x-2h_j(n)t}{\sqrt{1+4t^2}}\right).
	$$
\end{theorem}
\begin{proof}
We a similar arguments used in Theorem \ref{Fourier1}. We apply the Fourier transform to the ODE in \eqref{star4} and we get
	\begin{equation}
		\label{eqH}
		\widehat{\phi_{k,n}}(\lambda,t)= c(\lambda) e^{-i \lambda^2 t}.
	\end{equation}
	Since $c(\lambda)= \hat{\psi}(\lambda, 0)$ and the Hermite functions are eigenfunctions functions of the Fourier transform, i.e. $\mathcal{F}[\mathbf{h}_k](\lambda) = \widehat{\mathbf{h}_k}(\lambda) =\sqrt{2\pi} (-i)^k \mathbf{h}_k(\lambda)$, we have
	\begingroup\allowdisplaybreaks
	\begin{align}
		\nonumber
		c(\lambda)  =  \int_{\mathbb{R}} e^{-i \lambda x}\phi_k(x,0)dx
		&=\int_{\mathbb{R}} e^{-i\lambda x } \mathbf{h}_k(x)f_n(x)dx\\
		\nonumber
		&=\sum_{j=0}^{n} Z_j(n,a) \int_{\mathbb{R}} e^{-ix \left(\lambda -h_j(n)\right)}\mathbf{h}_k(x)dx\\
		\nonumber
		&=\sum_{j=0}^{n} Z_j(n,a) \mathcal{F}[\mathbf{h}_k]\left(\lambda-h_j(n)\right)\\
		\label{con1}
		&=\sqrt{2\pi}\sum_{j=0}^{n} Z_j(n,a) (-i)^k\mathbf{h}_k\left(\lambda-h_j(n)\right).
	\end{align}
\endgroup
	By plugging \eqref{con1} into \eqref{eqH}, using the definition of the Hermite functions i.e. $ h_k(x) = e^{-x^2/2}H_k(x)$, and inverse Fourier transform we have
		\begingroup\allowdisplaybreaks
	\begin{align*}
		\phi_k(x,t) & = \frac{1}{{2\pi}} \int_{\mathbb{R}} e^{i\lambda x} \hat{\phi_k}(\lambda , t ) d\lambda \\
		&=\frac{1}{\sqrt{2\pi}} \int_{\mathbb{R}} e^{i\lambda x}
		\left(
		e^{ -i\lambda^2t }  
		\sum_{j=0}^{n} Z_j(n,a) 
		(-i)^k \mathbf{h}_k(\lambda-h_j(n) )
		\right)
		d\lambda\\
		& =\frac{(-i)^k}{\sqrt{2\pi}} \sum_{j=0}^{n} Z_j(n,a) \int_{\mathbb{R}} e^{i\lambda x -i\lambda^2 t} \mathbf{h}_k (\lambda- h_j(n)) d\lambda\\
		&=\frac{(-i)^k}{\sqrt{2\pi}} \sum_{j=0}^{n} Z_j(n,a) \int_{\mathbb{R}} e^{i\lambda x - i \lambda^2 t} e^{ -\frac{(\lambda-h_j(n) )^2}{2} } H_k(\lambda- h_j(n) )d\lambda\\
		&=\frac{(-i)^k}{\sqrt{2\pi}} \sum_{j=0}^{n} Z_j(n,a) e^{-\frac{h_j^2(n)}{2}} \int_{\mathbb{R}}
		e^{-(it+\frac{1}{2})\lambda^2 + \lambda (h_j(n)  + ix)}
		H_k(\lambda - h_j(n) ) d\lambda.
	\end{align*}
\endgroup
By using \eqref{inte2} with $a:=it+\frac{1}{2}$, $b:=h_j(n)  + ix$ and $c:=h_j(n)$ we get the result.
\end{proof}

\begin{remark}
	Observe that if $k=0$ we get that $H_0(x)=1$, and we get back the result of Theorem \ref{Fourier1}.
\end{remark}

\section*{Acknowledgments}

Antonino De Martino was supported by MUR grant “Dipartimento di Eccellenza 2023-2027". The research of Kamal Diki is supported by the Research–Flanders (FWO) under grant number 1268123N.

\end{document}